\documentclass[leqno,12pt]{amsart}
\usepackage{a4,latexsym,amssymb,amsfonts}

\usepackage{amsmath,amscd} 

\usepackage{enumerate}
\usepackage{amsthm}
\numberwithin{equation}{section}

\newtheorem{theorem}{Theorem}[section]
\newtheorem{lemma}[theorem]{Lemma}
\newtheorem{proposition}[theorem]{Proposition}
\newtheorem{corollary}[theorem]{Corollary}
\newtheorem{definition}[theorem]{Definition}

\theoremstyle{definition}
\newtheorem{remark}[theorem]{Remark}
\newtheorem{example}[theorem]{Example}

\newcommand\uu{u_{1}}
\newcommand\ud{u_{2}}
\newcommand\vu{v_{1}}
\newcommand\vd{v_{2}}
\newcommand{\ds}{\displaystyle}

\newcommand{\sj}{s_{j} }

\newcommand\RR{{\mathbb{R}}}
\newcommand\CC{{\mathbb{C}}}
\newcommand\NN{{\mathbb{N}}}
\newcommand\ZZ{{\mathbb{Z}}}

\newcommand\cI{{\mathcal{I}}}

\newcommand\cS{{\mathcal{S}}}
\newcommand\ciC{{\mathcal{C}}}

\newcommand\unme{\frac{1}{2}}
\newcommand\xu{x_{1}}
\newcommand\xd{x_{2}}
\newcommand\xt{x_{3}}

\newcommand\xiu{\xi_{1}}
\newcommand\xid{\xi_{2}}

\newcommand\psiarg{\psi\left(\xuxd,\xuxt\right)}

\newcommand\xuxd{\frac{\xd}{\xu^{m}}}
\newcommand\xuxt{\frac{\xt}{\xu^{n}}}

\newcommand\sgn{\text{sgn}}
\newcommand\Sch{{\mathcal{S}}}
\newcommand\lime{\ds\lim_{{\varepsilon\to 0^+}}}

\newcommand\hpu{\overline{p}_{1}}
\newcommand\hpd{\overline{p}_{2}}
\newcommand\tpu{\tilde{p}_{1}}
\newcommand\tpd{\tilde{p}_{2}}
\newcommand\insieme{\RR^2\setminus\{(0,0)\}}

\begin{document}

\title
{Product kernels adapted to curves in the space}
\author{ Valentina Casarino}
\author{Paolo Ciatti}
\author{Silvia Secco}
\address{ Dipartimento di Matematica\\
Politecnico di Torino\\
Corso Duca degli Abruzzi 24\\10129 Torino}
\address{ 
Dipartimento di Metodi e Modelli matematici per le scienze applicate\\
Via Trieste 63, 35121 Padova}
\email{valentina.casarino@polito.it, ciatti@dmsa.unip.it, silvia.secco@polito.it}
\thanks{}
\date{\today}

\maketitle

\begin{abstract}
We establish  $L^p$-boundedness
for a class of operators that are given by convolution with product kernels adapted to
 curves in the 
space. 
The $L^p$ bounds follow from the decomposition of the adapted kernel 
into a sum
of two kernels with sigularities
concentrated respectively
on a coordinate plane and along 
the curve.

The proof of the $L^p$-estimates  for the two corresponding operators involves
Fourier analysis techniques  and
some algebraic tools, namely
 the Bernstein-Sato polynomials.
  \end{abstract}

\section{Introduction}
\medskip
The purpose of this paper is to establish  $L^p$ boundedness for a class of
product-type  convolution operators.

In the last thirty years the theory of 
singular integrals on product domains
has been largely developed.
The first    case which was considered is that
of a convolution operator $Tf=K*f$ on $\RR^{d_1}\times \RR^{d_2}$,
with $K(x,y)=K_1(x)K_2 (y)$,
$x\in\RR^{d_1}$, $y\in\RR^{d_2}$,
 $K_1$ and $K_2$ being of Calder\'on-Zygmund type. 
 In this case  a simple iteration argument yields the $L^p$ boundedness of $T$.
A more involved situation is that of 
a convolution operator $T$, whose kernel $K$ 
is defined on
$\RR^{d_1}\times \RR^{d_2}$
and satisfies all the analogous bounds 
to those 
satisfied by 
$K_1 K_2$ on 
$\RR^{d_1}\times \RR^{d_2}$, but cannot be decomposed
as product of two kernels
$K_1(x)$ and $K_2(y)$.
A precise definition of such kernels, which are called "product kernels",
was introduced in terms
of  certain differential inequalities and
suitable cancellation conditions.

Several conditions on $K$, guaranteeing the $L^p$
boundedness of the operator $T$, have been introduced \cite{FSt}, 
and 
many applications of  the product theory to the operators arising in certain boundary value problems
have been studied \cite{NagelRicciStein}, \cite{NagelStein2}.
.
Moreover, the euclidean spaces 
$ \RR^{d_j}$, $j=1,2$, have been replaced
by appropriate nilpotent groups \cite{MRS}, \cite{NagelRicciStein}, and by smooth
manifolds with a geometry determined by a control distance \cite{NagelStein1}.

Recently, one of the authors studied the $L^p$ boundedness for  convolution operators
with kernels obtained adapting product kernels to curves in the plane \cite{Secco}.
Here we extend 
these results to  higher dimensional spaces.
In order not to burden the exposition with 
notational complexities, we are going to give the full details
only for the space $\RR^3$, with $d_1=1$ and $d_2=2$.
In the last section we shall quickly 
describe how the arguments should be modified in the higher dimensional setting.

We denote an element of $\RR^3=\RR\times\RR^2$ by the pair $(x_1,x)$, where
$x=(x_2,x_3) \in \RR^2$.
On $\RR$ we consider the usual  dilations by $\delta>0$, while on
$\RR^2$ we consider the anisotropic dilation given by
\begin{equation}  \label{dilations_in_R^2}
\delta \circ x = (\delta^\frac{1}{2n}x_2,\delta^\frac{1}{2m}x_3)\,,\qquad {\text{with}}\quad
\delta>0\,,\,m,n\in\NN\,,\, m<n\,.
\end{equation}
We denote by 
\begin{equation}\label{dimomog}
Q=\frac{1}{2n}+\frac{1}{2m} 
\end{equation}
 the homogeneous
dimension of $\RR^2$ with respect to the dilations 
(\ref{dilations_in_R^2})
 and by $\rho(x) = x_2^{2n} +x_3^{2m} $ a smooth
homogeneous norm on $\RR^2$.

In this context the proto-typical example of a product kernel in
$\RR^3$ (we refer to  Section 3 for a precise definition)  is given by the distribution
\begin{equation}\label{nucleoH}
H(x_1,x) = C_{\mu}\,p.v.\frac{1}{x_1}\rho(x)^{-Q+i\mu}\,,
\qquad\mu\in\RR\setminus \{ 0 \}\,.
\end{equation}

Throughout the paper we concentrate our attention on 
the convolution by  product-type kernels in $\RR^3$ whose
singularities are supported 
on a coordinate plane and on a
transversal curve of finite type.
A rather simple example of such a kernel is
$$
K(x_1,x) = C_{\mu}\,p.v.\left(\frac{1}{x_1}\right)\rho(x-\gamma(x_1))^{-Q+i\mu}\,,\text{ $\mu\in\RR \setminus \{ 0 \}$}\,,
$$
where
$\gamma:\RR\to\RR^2$
is the curve $\gamma(x_1)=(x_1^m,x_1^n)$.

More generally, we introduce  the following class of product-type kernels.
\begin{definition}
Assume that $K_0$ is a product kernel on $\RR^3$ and consider the
curve $x=\gamma(x_1)$ with $\gamma(x_1)= (x_1^m,x_1^n)$, $x_1 \in
\RR$. We define a distribution $K$ by
\begin{equation} \label{adapted_product_kernel_def}
\displaystyle \int K(x_1,x)f(x_1,x)\,dx_1dx := \int
K_0(x_1,x)f(x_1,x+\gamma(x_1))\,dx_1dx
\end{equation}
for a Schwartz function $f$ on $\RR^3$.
$K$ will be called an adapted kernel.

Here with an abuse of notation we write pairings between
distributions and test functions as integrals.
\end{definition}

The kernel $K$ given by the formula
(\ref{adapted_product_kernel_def}) is a well-defined tempered
distribution which is singular on the coordinate plane $x_1=0$ and
along the curve $x=\gamma(x_1),$ $x_1\in \RR$.

We shall prove the following theorem.

\begin{theorem}\label{L^p_boundedness}
Let $K$ be the distribution defined by the formula
(\ref{adapted_product_kernel_def}). Then the 
convolution operator 
 $T: f\mapsto
f\ast K$, initially defined on the Schwartz space $\Sch (\RR^3)$, extends to a bounded
operator on $L^p(\RR^3)$ for $1<p<\infty$.
\end{theorem}

To prove Theorem \ref{L^p_boundedness} 
we decompose the adapted kernel $K$
of $T$ as the sum of a kernel $K_1$
 with singularities concentrated on the coordinate plane $x_1=0$
and of a kernel $K_2$ 
singular 
along the curve $x=\gamma(x_1),$ $x_1\in \RR$. As in \cite{Secco}
we show that the multiplier associated with $K_1$ 
satisfies some Marcinkiewicz-type  conditions, while $K_2$
is treated by means of analytic interpolation (our proof is inspired by some arguments used in 
\cite{SteinWainger} to prove the $L^p$ boundedness of the Hilbert transform
along curves in the plane).
In particular, to apply the analytic interpolation method
we need to introduce a non-isotropic version of the Riesz potentials
$$\cI^z (\uu,\ud):= \left( \rho(\uu,\ud)\right)^{z-Q}\,,
\qquad z\in\CC\,.$$
and to determine their 
meromorphic continuation.
To extend in a  meromorphic way 
$\cI^z $,  we  study the location of its singularities using
 Bernstein-Sato polynomials.
Since the reader is not assumed to be familiar with Bernstein-Sato functional identities, 
we illustrate the definition 
and the basic properties of this algebraic tool in Section 2.

In the last section we shall discuss how $L^p$ bounds for convolution by product kernels adapted to curves
are related  in a natural way  to the study of $L^p-L^q$ estimates 
for analytic families of fractional operators \cite{CCiSe}.

Throughout the paper we will use the Òvariable constant conventionÓ, and denote by C, possibly with
sub- or superscripts, a constant that may vary from place to place.

\section{ Bernstein-Sato polynomials and a family of Riesz-type kernels}
Consider the polynomial
\begin{equation}\label{polinomio}
\rho (\uu,\ud):=\uu^{ 2n}+\ud^{2m}\,,
\end{equation}
with 
$m,n\in\NN\,,\,\, m\ge1\,,\,\, n>m\,$. Observe that $\rho$ is
homogeneous with respect to the one-parameter family of non
isotropic dilations
 given by
(\ref{dilations_in_R^2}).

We shall often use, in the following, 
the relations between $\rho$ and  the euclidean norm  $|\cdot|$
 in $\RR^2$
\begin{equation}\label{stimerho}
A\rho (u)^{{\frac{1}{2n}}}\le |u|\le B\rho (u)^{{\frac{1}{2m}}}
\;\;\;\text{ if $\rho(u)>1$}
\end{equation}
and 
\begin{equation}\label{stimerho1}
A'\rho (u)^{{\frac{1}{2m}}}\le |u|\le B'\rho (u)^{{\frac{1}{2n}}}
\;\;\;\text{ if $\rho(u)\le 1$}
\end{equation}
for some $A, B, A', B'>0$.

Define now the distribution
\begin{equation}\label{definizione Iz}
\cI^z (\uu,\ud):= \left( \rho(\uu,\ud)\right)^{z-Q}\,,
\end{equation}
where $z\in\CC$, $\Re e z>0$.

Observe that both $\rho$ and $\cI^z$ depend on $m$ and $n$.
Anyway, for the sake of semplicity,
we shall avoid to indicate the dependence on $m$ and $n$.

We collect in the next proposition some obvious properties of $\cI^z$.

\begin{proposition}
If 
$\Re e z
>0$,
then 
\item i) $\cI^z$ is well defined as distribution and locally integrable.

\item ii) $\cI^z$ is a tempered distribution.

\item iii) $\cI^z$ is an analytic family of tempered distributions, that
is, given $f\in\cS(\RR^2)$, the functions $z\mapsto<\cI^z, f>$ are
holomorphic.
 \end{proposition}

We shall now prove that $\cI^z$ admits a meromorphic extension, with
poles in a at most countable set, consisting of rational negative
points. Our method is based on the theory of  Bernstein-Sato polynomials.

It is well-known in algebra that, given a non-zero polynomial $p
(\uu,\ud)$ with complex coefficients, there exist a non-zero
polynomial $b_{p}(s)\in\CC[s]$ and a differential operator $L(s)$
whose coefficients are polynomials in $s\,,\uu\,,\ud\,$, such
that formally
\begin{equation}
\label{BS} L(s)\left( p (\uu,\ud)\right)^{s+1}= b_{p}(s)\left( p
(\uu,\ud)\right)^{s}\,\text{ for all $s\in\CC$}\,.
\end{equation}

The set of all polynomials $b_{p}(s)\in\CC[s]$ satisfying this
formal identity (for some operator $L$) is an ideal, and the
unique monic generator of this ideal is called the {\it{Bernstein-Sato
polynomial }} of $p$.

In our case, for $\Re e z>0$ we may write
\begin{equation}
\label{BSzQ} \left( \rho(\uu,\ud)\right)^{z-Q}=\frac{ L(z-Q)\left(
\rho(\uu,\ud)\right)^{z+1-Q}}{ b_{\rho}(z-Q)}.
\end{equation}
By repeatedly using the functional equation ($\ref{BSzQ}$) we may
extend $\cI^z$ to the complex plane in a meromorphic way, with poles
whenever $b_{\rho}(z-Q+k)$ vanishes for a non-negative integer $k$.
Therefore we shall now seek for the zeros of the Bernstein-Sato
polynomial $b_{\rho}(z-Q)$.
\par
According to a theorem of Kashiwara, the roots of the Bernstein-Sato
polynomial are negative rational numbers. 
Moreover, if $\rho $ has the particularly simple form given in
($\ref{polinomio}$), it is easy to find the roots of $b_{ \rho}(s)$.
\begin{lemma}\label{ lemmaradici}
If $b_{ \rho}(s)$ denotes the Bernstein-Sato polynomial associated
to $\rho (\uu,\ud):= \uu^{ 2n}+\ud^{ 2m}$, $m<n$, $m,n\in\NN$, then
the roots of $\frac{ b_{ \rho}(s)}{ s+1}$ are given by
\begin{equation}\label{radiciKa}
-\frac{p_1 }{2n }-\frac{p_2 }{2m }\,, \;\; 1\le p_1\le 2n-1\,, \quad
1\le p_2\le 2m-1\,,
\end{equation}
with multiplicity one.
\end{lemma}
\begin{proof}
It is essentially due to Kashiwara \cite{Kashiwara}. See also \cite{Malgrange}
 and [BMSa, remark
3.8].
\end{proof}
 In the following corollary we collect some observations, which will be
useful in the following.
\begin{corollary}
i) The largest root of $b_{\rho}(s)$ is $-Q=-
\frac{1}{2m}-\frac{1}{2n}
$.
\par\noindent
ii) $-1$ is a root of $b_{\rho}(s)$ with multiplicity two.
\par\noindent
iii) The set of the roots of $b_{\rho}(s)$ is symmetric with respect
to $-1$.
\end{corollary}
\begin{proof}
i) It follows obviously from ($\ref{radiciKa}$) for $p_1=1$ and
$p_2=1$.
 \par\noindent
 ii) Observe that
$-\frac{p_1 }{2n }-\frac{p_2 }{2m }=-1$ for ${p_1 }={n }$ and ${p_2
}={m }$. Then $-1$ is a root of multiplicity one for
$\frac{b_{\rho}(s)}{s+1}$, whence the thesis follows.
\par\noindent
iii) Suppose that the $-1+\delta:=-\frac{\hpu }{2n }-\frac{\hpd }{2m
}$ is a root of $b_{\rho}(s)$ for some $\hpu\,, \hpd\in\NN$, $1\le
\hpu\le 2n-1\,, \, 1\le \hpd\le 2m-1\,$
 and some $\delta>0$.
Take
 $\tpu:=2n-\hpu$ and $\tpd:=
 2m-\hpd$. Since
$1\le \tpu\le 2n-1\,, \, 1\le \tpd\le 2m-1\,$, 
then
 $$ -\frac{\tpu }{2n }-\frac{\tpd }{2m }
=-2 +\frac{\hpu }{2n }+\frac{\hpd }{2m }= -1-\delta\,$$
is a root of $b_{\rho}(s)$. 
\end{proof}

\begin{example}\label{ esempio6-4}

By means of formula ($\ref{radiciKa}$) it is possible to find the
roots of the Bernstein-Sato polynomial associated to $\uu^{
6}+\ud^{4}\,$ and one finds
\begin{align*}
b_{\rho}(s)=&(s+1)^2\, (s+\frac{2}{3})\, (s+\frac{4}{3})\,
(s+\frac{3}{4})\, (s+\frac{5}{4})\, 
(s+\frac{5}{6})\,
(s+\frac{7}{6})\, 
(s+\frac{5}{12})\,
(s+\frac{19}{12})\,\\
&
(s+\frac{7}{12})\,
(s+\frac{17}{12})\,
(s+\frac{11}{12})\, (s+\frac{13}{12})\, 
 \,.
\end{align*}
 \vspace*{0.1in}
\end{example}

Let us now consider the distribution $\cI^z$ defined by
(\ref{definizione Iz}). Let $s_{1}\,,\ldots, s_{h}$ be the zeros of
$b_{\rho }(s)$ in $(-Q-1, -Q]$,
each counted with its multiplicity
and ordered in a decreasing way.
Then the meromorphic continuation of  $\cI^z$ 
has poles whenever $z=Q+s_{j}-k$, $k\ge 0$. We
remark, in particular, that $0$ and $-1+Q$ are always poles for
$\cI^z$. By gluing all together, we get the following result.
\begin{proposition}
$\cI^z$ may be analytically continued to a meromorphic
distribution-valued function of $z$, also denoted by $\cI^z$,  with poles in a set $A$,
consisting of rational negative points.
 More precisely,
 $$A=\left\{
 \zeta_{j,k}:=Q+s_{j}-k
\,: k\in\NN\,,\, j=1,\ldots, h \right\},$$ $\sj$, $j=1,\ldots, h$,
denoting the zeros of the Bernstein-Sato polynomial $b_{\rho}$ in
$(-Q-1, -Q]$,
each listed as many times as its multiplicity. Each pole has order
one, with the exception of the points $-1+Q-k$, $k\in\NN$, which have order two.
\end{proposition}

Set $\zeta_j:=\zeta_{j,0}$. 
Observe
that $\zeta_1=0$ is a pole of order $1$ for $\cI^z$.

Consider now the function $G$, given by
\begin{equation}\label{GMN}
G(z):=
\Gamma\left( z+1-Q\right)\cdot
\prod_{ j=2,..,h}\!\Gamma
(z-\zeta_j)\,.
\end{equation}
If $S$ denotes the sphere
$$S:=\{(\uu,\ud)\in\RR^2\,:\, \rho (\uu,\ud)=1
\}$$ with surface measure $\sigma (S)$, set
\begin{equation}\label{IMN}
{I}^z(\uu,\ud):=\frac{G(0)
\left( \uu^{2n}+\ud^{2m}\right)^{z-Q}}{\sigma(S)\,\Gamma (z)
G(z)}\,.
\end{equation}

In the sequel  we will denote by $\cS(\RR^s)$, $s=2,3$, the
Schwartz space on $\RR^s$ endowed with a denumerable family of
norms $\|\cdot\|_{(N)}$ given by
$$
\|\Phi\|_{(N)} = \sum_{|\alpha|\leq N} \sup_{u\in
\RR^s}(1+|u|)^N|\partial_u^\alpha \Phi (u)|.
$$
Here we use the conventional notation
$$
\partial_u^\alpha
=\frac{\partial^{\alpha_1}}{\partial u_1^{\alpha_1}}\cdots
\frac{\partial^{\alpha_s}}{\partial u_s^{\alpha_s}} 
\,,$$
with 
$\alpha=(\alpha_1,\dots,\alpha_s)$  $s$-tuple of natural
numbers
and 
$|\alpha|=\alpha_1+\cdots +\alpha_s$.
\begin{proposition}\label{deltaDirac}
The distribution $I^z$ satisfies
$$I^0=\delta_0\,.$$
\end{proposition}
\begin{proof}
Take $z\in\CC$, $\Re e z>0$, and $\varphi\in\cS(\RR^2)$. Set $C_z:=
\frac{G(0)}{\sigma(S) \Gamma(z)G(z)}$. Then
\begin{align*}
<I^{z},\varphi >&= C_z
 \int_{\RR^2}
\rho (\uu,\ud)^{z-Q} \varphi (\uu,\ud)\,d\uu\,d\ud
\\
&= C_z \left(
 \int_{
\{\rho(\uu,\ud)\le 1\}
 }
\rho (\uu,\ud)^{z-Q} \left( \varphi (\uu,\ud)-\varphi (0,0)\right)
\,d\uu\,d\ud\right.\\
&\left. +
 \int_{\{\rho(\uu,\ud)\le 1\} }
\rho (\uu,\ud)^{z-Q} \varphi (0,0) \,d\uu\,d\ud +
 \int_{
\{\rho(\uu,\ud)\ge 1\}
 }
\rho (\uu,\ud)^{z-Q} \varphi (\uu,\ud)\,d\uu\,d\ud
\right)\\
&= C_z \left( I_1+ I_2+ I_3\right)\,. \qquad \qquad \qquad \qquad
\qquad \qquad \qquad \qquad \qquad \qquad \qquad \qquad \qquad
\\
\end{align*}
By introducing polar coordinates (see \cite{FoSt}) we obtain
\begin{align*}
I_2=&
 \int_S
 \int_0^1
 \varphi (0,0)
\rho (r\circ (\vu,\vd))^{z-Q} r^{Q-1}dr\,d\sigma (\vu\,\vd)=
 \int_S
 \int_0^1
 \varphi (0,0)\left(
r\cdot \rho ( \vu,\vd)\right)^{z-Q}
r^{Q-1}dr\,d\sigma (\vu\,\vd)\\
=&
 \varphi (0,0)
\int_S
 \int_0^1
 r^{z-1}dr\,d\sigma (\vu\,\vd)
=
 \varphi (0,0)
\frac{\sigma (S)}{z}\,,
\end{align*}
so that
\begin{equation}\label{gamma z+1}
C_z {I_2} =
 \varphi (0,0)
\frac{G(0)}{ z\Gamma (z)\, G(z)}=
 \varphi (0,0)
\frac{G(0)}{ \Gamma (z+1)\, G(z)}\,
\end{equation}
and this expression is well-defined for every $z$, with $\Re e
z>-\min\{-\zeta_2,1\}$.

Now, it it is easy to show that both $I_1$ and $I_3$ are absolutely
convergent for $\Re e z>-\min\{\frac{1}{2n},-\zeta_2\}$. Indeed,
\begin{align*}
| I_1|\le&
  \int_{
\{\rho(\uu,\ud)\le 1\}
 }
\rho (\uu,\ud)^{\Re e z-Q} \left| \varphi (\uu,\ud)-\varphi
(0,0)\right| \,d\uu\,d\ud
\\
\le&C||\nabla\varphi||_{\infty}\int_{ \{\rho(\uu,\ud)\le 1\}
 }
\rho (\uu,\ud)^{\Re e z-Q} \left| \left( \uu,\ud\right)\right|
 \,d\uu\,d\ud
\\
\le&C||\nabla\varphi||_{\infty}\int_{ \{\rho(\uu,\ud)\le 1\}
 }
\rho (\uu,\ud)^{\Re e z-Q+\frac{1}{2n}} \,d\uu\,d\ud
\\
= &C ||\nabla\varphi||_{\infty} \int_{S}\int_{0}^1 \rho \left(
r\circ (v_1,v_2)\right)^{\Re e z-Q+\frac{1}{2n}} \,
r^{Q-1}\,dr\,d\sigma(\vu,\vd)
\\
= &C ||\nabla\varphi||_{\infty} \int_{S}\int_{0}^1 r^{\Re e
z-Q+\frac{1}{2n}} \, r^{Q-1}\,dr\,d\sigma(\vu,\vd)
\\
=&C ||\nabla\varphi||_{\infty}
\frac{\sigma(S)}{\frac{1}{2n}+\Re e z},\\
\end{align*}
which is well-defined for $\Re e z>-\frac{1}{2n}$. 
Here, in particular, we used 
(\ref{stimerho1}).

Moreover,
\begin{align*}
| I_3|\le&
  \int_{
\{\rho(\uu,\ud)> 1\}
 }
\rho (\uu,\ud)^{\Re e z-Q} \left| \varphi (\uu,\ud)\right|
\,d\uu\,d\ud
\\
\le&C||\varphi||_{(N)}\int_{ \{\rho(\uu,\ud)> 1\}
 }
\rho (\uu,\ud)^{\Re e z-Q-\frac{N}{2n}}
\,d\uu\,d\ud\,,\\
\end{align*}
since
$$|\varphi (\uu,\ud)|\le
\frac{||\varphi||_{(N)}}{\left( 1+|(\uu,\ud)| \right)^N} \le
\frac{||\varphi||_{(N)}}{ |(\uu,\ud)|^N } \le C\,
\frac{||\varphi||_{(N)}}{ \left(\rho(\uu,\ud)\right)^{\frac{N}{2n}}
}\,$$ when $\rho (\uu,\ud)>1$, as a consequence of
(\ref{stimerho}). Now by passing to polar coordinates
we obtain
\begin{align*}
| I_3|\le& C||\varphi||_{(N)} \int_S \int_1^{+\infty} \rho \left(
r\circ (\vu,\vd)\right)^{\Re e z-Q- \frac{N}{2n} } r^{Q-1}\,dr\,
\,d\sigma(\vu\,\vd)\,\\
=&C ||\varphi||_{(N)} \frac{\sigma(S)}{\frac{N}{2n}-\Re ez}<+\infty
\,,\\
\end{align*}
if $N$ is a positive integer greater than ${2n}\cdot \Re e z$. Thus,
as a consequence of the uniqueness of the analytic continuation, the
expression
$$<I^{z},\varphi >=
C_z \left( I_1+ I_2+ I_3\right)\,$$ defines the action of $I^z$ on a
Schwartz function $\varphi$ in $\RR^2$, for $\Re e z> -
\min\{{\frac{1}{2n}, -\zeta_2,1}\}$ and by using the bounds for
$I_1$ and $I_3$ and (\ref{gamma z+1}) one gets the thesis.

\end{proof}

\begin{proposition}
${I^z}$ is a homogeneous tempered distribution of
degree $-Q+z$.
\end{proposition}
We recall that this means that for all $\varphi\in\cS(\RR^2)$ the
following equality is satisfied
$$<I^z,\varphi_{\delta}>=
\delta^{z-Q}\,<I^z,\varphi>\,,$$ where
$$\varphi_{\delta}(\uu,\ud):=
\delta^{-Q}\varphi\left( \delta^{-1}\circ u \right)=
\delta^{-Q}\varphi\left( \delta^{-\frac{1}{2n}} \uu,
\delta^{-\frac{1}{2m}} \ud\right)\,.$$ Thus the Fourier transform of
the (tempered and homogeneous) distribution $I^z$ is a
well-defined distribution, homogeneous of degree $-Q- (z-Q)=-z$.
Moreover, the following holds.
\begin{proposition}\label{FouriertransformI^z}
$\widehat{I^z}$ agrees with a    function  $C^{\infty}(\RR^2\setminus\{(0,0)\})$
away from $(0,0)$.
Moreover,
\begin{equation}\label{decadimentotrasfF}
|\widehat{I^z}(\xi)|\le C\rho(\xi)^{-\Re ez}\,,
\end{equation}
for all $\xi\in\RR^2\setminus\{(0,0)\}$.
\end{proposition}

\begin{proof}
It suffices to prove the statement for
$0<\Re e z<Q$; indeed,  the other cases
can be treated by analytic continuation.

 We first  construct a partition of unity adapted to the dyadic spherical shells.
 The procedure is standard and we briefly recall it only for the
 sake of completeness.
 
Let $\psi$ be a $C^{\infty}_{c}(\RR^2)$ function, such that
\begin{item}
\item (i) $0\le \psi (\uu,\ud)\le 1 $ for every $(\uu,\ud)\in\RR^2$;
\item (ii) $\psi(\uu,\ud)\equiv 0$ if $(\uu,\ud)\not\in
C_0:=
\{
(\uu,\ud)\in\RR^2\,:\,
{\frac{1}{4}}\le \rho (\uu,\ud)\le 8
\}$\,;
\item (iii)
$\psi(\uu,\ud)\equiv 1$ if $(\uu,\ud)\in
C_1:=
\{
(\uu,\ud)\,:
\, \frac{1}{2}\le \rho(\uu,\ud)\le 4
\}$.
\end{item}


Define now for $(\uu,\ud)\in\insieme$
\begin{equation}\label{Psi}
\Psi (\uu,\ud):=
\sum_{j\in\ZZ}
\psi (2^{j}\circ (\uu,\ud))\,.
\end{equation}
Since  there is 
at most a finite number of nonzero terms in the sum
$(\ref{Psi})$,  $\Psi$ is well-defined 
and strictly positive on $\insieme$.
Thus we may introduce 
the functions
\begin{equation}\label{eta}
\eta (\uu,\ud):=
\frac{\psi  (\uu,\ud)}{\Psi (\uu,\ud)}\,.
\end{equation}
It is easy to check that 
\begin{equation}\label{sommaetaj}
\sum_{j\in\ZZ}
\eta \big(2^j\circ (\uu,\ud)\big)=1
\,\,\,{\rm{for \,\,every }}\, \,(\uu,\ud)\in\insieme
\end{equation}
Now using  ($\ref{sommaetaj}$)
we may write
\begin{align*}
I^z (\uu,\ud)&=
C_z \rho (\uu,\ud)^{z-Q}
\\
&=C_z \sum_{j\in\ZZ}
\eta( 2^{j}\circ (\uu,\ud))
\rho(
2^{-j}\circ 2^j
\circ (\uu,\ud))^{z-Q}\\
&=C_z \sum_{j\in\ZZ}
\eta( 2^{j}\circ (\uu,\ud))
2^{-j(z-Q)}
\rho(
 2^j
\circ (\uu,\ud))^{z-Q}\\
&=C_z \sum_{j\in\ZZ}
2^{-j(z-Q)}
f_0 ( 2^{j}\circ (\uu,\ud))\,,\\
\end{align*}
where we set
$$f_0 (\uu,\ud):=\eta(\uu,\ud)
\rho (\uu,\ud)^{z-Q}\,.$$
Since
$$
\left(f_0 (2^j \circ (\cdot,\cdot))
\right)\widehat{}\,(\xiu,\xid)
=
2^{-jQ}
\widehat{f_0} \left( 2^{-j}\circ (\xiu,\xid)\right)\,,$$
we obtain formally
$$
 \sum_{j\in\ZZ}
2^{-j(z-Q)}
\left(
f_0 ( 2^{j}\circ (\cdot,\cdot))\right)\widehat{}\,(\xiu,\xid)
=
 \sum_{j\in\ZZ}
2^{-jz}
\widehat{ f_0 }( 2^{-j}\circ (\xiu,\xid))\,.
$$
This series  is absolutely convergent, since 
if $(\xiu,\xid)\neq (0,0)$ one has
\begin{align*}
\left|
\sum_{j\in\ZZ}
2^{-jz}
\widehat{ f_0 }( 2^{-j}\circ (\xiu,\xid))\right|
&\le
\sum_{j\in\ZZ}
2^{-j\Re e z}
\left|
\widehat{ f_0 }( 2^{-j}\circ (\xiu,\xid))\right|\\
&\le
\left(
\sum_{2^{-j}\rho (\xiu,\xid)\le 1}
+
\sum_{2^{-j}\rho (\xiu,\xid)> 1}
\right)
2^{-j\Re e z}
\left|
\widehat{ f_0 }( 2^{-j}\circ (\xiu,\xid))\right|\\
&\le
||\widehat{ f_0 }||_{ (0)}
\sum_{2^{-j}\rho (\xiu,\xid)\le 1}
2^{-j\Re e z}
+
\sum_{\rho
( 2^{-j}\circ (\xiu,\xid))> 1}
2^{-j\Re e z}
\frac{||\widehat{f_0}||_{(N)}}{\left(1+\left| \left(
2^{-j}\circ (\xiu,\xid)\right)\right|\right)^N
}
\\
&\le
C\,||\widehat{ f_0 }||_{ (0)}
\rho(\xi)^{-\Re e z}
+
C\,
\sum_{
 2^{-j}\rho (\xiu,\xid)> 1}
2^{-j\Re e z}
\frac{||\widehat{f_0}||_{(N)}}{\left(
\rho( 2^{-j}\circ (\xiu,\xid))\right)^{\frac{N}{2n}}
}\\
&\le
C\,||\widehat{ f_0 }||_{ (0)}
\rho(\xi)^{-\Re e z}
+
\frac{C}{
\left(
\rho(  \xiu,\xid)\right)^{\frac{N}{2n}}
}
\sum_{
 2^{-j}\rho (\xiu,\xid)> 1}
2^{-j\Re e z+j\frac{N}{2n}}
\\
&\le
C\,||\widehat{ f_0 }||_{ (0)}
\rho(\xi)^{-\Re e z}
+
\frac{C}{
\left(
\rho(  \xiu,\xid)\right)^{\frac{N}{2n}}
}
\rho(\xiu,\xid)^{-\Re e z+\frac{N}{2n}}
\\
\\
&\le
C\,
\rho(\xiu,\xid)^{-\Re e z}\,,
\end{align*}
where we used in particular
 the fact that
$$|(\uu,\ud)|\ge C \,\rho (\uu,\ud)^{\frac{1}{2n}}\;\;
{\rm{for\,\,}} |(\uu,\ud)|>1\,.$$
We set therefore
$$v(\xiu,\xid):=
\sum_{j\in\ZZ}
2^{-jz}
\widehat{ f_0 }( 2^{-j}\circ (\xiu,\xid))\in\L^{\infty}(\RR^2)\,.$$
By the Dominated Convergence Theorem we obtain, given $\varphi\in \cS(\RR^2)$,
$$\int_{\RR^2}
v\varphi=
\sum_{j\in\ZZ}
2^{-jz}
\int_{\RR^2}
\widehat{ f_0 }( 2^{-j}\circ (\cdot,\cdot))\,\varphi\,,$$
that is
$$v(\cdot,\cdot)=
\sum_{j\in\ZZ}
2^{-jz}
\widehat{ f_0 }( 2^{-j}\circ (\cdot,\cdot))\,
$$
in the  sense of distributions, 
whence
$$\widehat{ I^z}(\cdot,\cdot)=
\sum_{j\in\ZZ}
2^{-jz}
\widehat{ f_0 }( 2^{-j}\circ (\cdot,\cdot))\,
$$
in the  sense of distributions
and, moreover,
\begin{equation}\label{stimatrasf}
\left|
\widehat{ I^z}(\xiu,\xid)\right|
\le\rho(\xiu,\xid)^{-\Re ez}\,\,{\rm{ for \,\, all\,\,}}{(\xiu\xid)\in\insieme}\,
\end{equation}
(observe that this inequality could also be retrieved
from the homogeneity).
Finally we prove that
 $\widehat{ I^z}$
 agrees with a    function  $C^{\infty}(\RR^2\setminus\{(0,0)\})$
away from $(0,0)$.
First of all, we observe that $f_0$ is in the Schwartz space, hence
$\widehat{ f_0 }( 2^{-j}\circ (\cdot,\cdot))$ belongs to $\ciC^{\infty }(\RR^2)$.
Moreover, the following estimates  hold:
\begin{align*}
\left|
\partial^{k}_{\xiu}\left(
2^{-jz}
\widehat{ f_0 }( 2^{-j}\circ (\xiu,\xid))
\right)\right|&\le
C_{k}\,2^{-j\left(\Re e z+\frac{k}{2n }  \right)}
{\rm{\,\,for\,\, all\,\,}}(\xiu,\xid)\in\RR^2
\\
\left|
\partial^{k}_{\xiu}\left(
2^{-jz}
\widehat{ f_0 }( 2^{-j}\circ (\xiu,\xid))
\right)\right|&\le
C_{k,N}\,2^{-j\left(\Re e z+\frac{k}{2n }  \right)}
\frac{1}{\left(
1+\left|\left(
2^{-j}\circ(\xiu,\xid)\right)\right|
\right)^N}
\\
&\le
\frac{C_{k,N}}{\rho(\xiu,\xid)^\frac{N}{2n}}\,{2^{j\left(-\Re e z-\frac{k}{2n }+\frac{N}{2n}  \right)}}
{\rm{\;\;\,\,\,\,if\,\,}}\rho\left(2^{-j}\circ(\xiu,\xid)\right)>1\,,\\
\end{align*}
for all $k\in\NN$.
Since analogous bounds
hold for
$
\partial^{k}_{\xid}
(
2^{-jz}
\widehat{ f_0 }( 2^{-j}\circ (\xiu,\xid))
$, 
with $2n$ replaced by $2m$,
the series of the partial
derivatives
of
$
2^{-jz}
\widehat{ f_0 }( 2^{-j}\circ (\cdot,\cdot))$
of any order $k$
converge on the compact subsets of $\RR^2\setminus\{(0,0)\}$.
It follows that
$\widehat{ I^z}(\cdot,\cdot)=
\sum_{j\in\ZZ}
2^{-jz}
\widehat{ f_0 }( 2^{-j}\circ (\cdot,\cdot))\,
$
is $\ciC^\infty (\RR^2\setminus\{(0,0)\})$
\end{proof}

\section{Some preliminary results}
In the following, if $f(x_1,x)\in \Sch(\RR^3)$ we denote by
${\mathcal{F}}^{-1}f$ the inverse Fourier transform of $f$ and by
${\mathcal{F}}_2f$ and ${\mathcal{F}}_2^{-1}f$ respectively the
partial Fourier transform and the  inverse of the partial Fourier transform
of $f$ with respect to the variable $x$.

Moreover
we
denote the dual variables as $(\xi_1,\xi)$ with $\xi=(\xi_2,\xi_3)$.

\subsubsection*{Characterization of product kernels}

As recalled in the Introduction,
the precise definition of product kernels involves certain
differential inequalities and certain cancellation conditions which
are analogous to those satisfied by the kernel $H(x_1,x)$ defined by
(\ref{nucleoH}).
 Our study will be based on the following equivalent definition
 (see  \cite{NagelRicciStein}).

\begin{definition}\label{product_kernel_def}
A product kernel $K$ on $\RR^3$ is a sum
\begin{equation}
K(x_1,x) = \sum_{J\in \ZZ^2}
2^{-j_1-jQ}\psi_J(2^{-j_1}x_1,2^{-j}\circ x), \qquad J=(j_1,j)
\end{equation}
convergent in the sense of distributions, of smooth functions
$\psi_J$ supported on the set where $1/2 \leq |x_1| \leq 4$ and $1/2
\leq \rho(x) \leq 4$, satisfying the cancellation conditions
\begin{equation} \label{x_1_cancellation}
\displaystyle \int \psi_J(x_1,x)\,dx_1=0\\
\end{equation}
\begin{equation} \label{x_cancellation}
\displaystyle \int \psi_J(x_1,x)\,dx=0\\
\end{equation}
identically for every $J$, and with uniformly bounded $C^k$ norms
for every $k\in \NN$.
\end{definition}

We shall need a characterization of product kernels as dyadic sums
of Schwartz functions on $\RR^3$ which are compactly supported only
in the first variable and that satisfy some moment conditions.

\begin{lemma}\label{dyadic_decomposition}
A product kernel $K$ on $R^3$ can be written as a sum
$$
K(x_1,x)  = \sum_{J\in \ZZ^2}
2^{-j_1-jQ}\varphi_J(2^{-j_1}x_1,2^{-j}\circ x), \qquad J=(j_1,j),
$$
convergent in the sense of distributions, of Schwartz functions
$\varphi_J$ such that
\begin{itemize}
\item[(i)] the $\varphi_J$ have compact $x_1$-support where $1/2 \leq
|x_1| \leq 4$;
\item[(ii)] the $\varphi_J$ form a bounded set of $\Sch(\RR^3)$, that is
the Schwartz norms $\|\varphi_J\|_{(N)}$ are uniformly bounded in
$J$ for each $N \in \NN$;
\item[(iii)] the $\varphi_J$ satisfy the cancellation conditions
\begin{equation} \label{x_1_moment}
\displaystyle \int x_1^{\ell}\varphi_J(x_1,x)\,dx_1=0\\
\end{equation}
for every positive integer  $\ell\le M_1$, for some fixed $M_1\in\NN$, and
\begin{equation} \label{x_moments}
\displaystyle \int x^\beta\varphi_J(x_1,x)\,dx=0\\
\end{equation}
for every multi-index $\beta = (\beta_1,\beta_2) \in \NN^2$,
identically for every $J\in \ZZ^2$. Here, as usual, $x^\beta =
x_2^{\beta_1} x_3^{\beta_2}$.
\end{itemize}
\end{lemma}

\begin{proof}
 Let $K$ be a product kernel on $\RR^3$. By Definition
\ref{product_kernel_def} we can write $K$ as a sum
$$
K(x_1,x)  = \sum_{(j_1,i)\in \ZZ^2}
2^{-j_1-iQ}\psi_{(j_1,i)}(2^{-j_1}x_1,2^{-i}\circ x),
$$
convergent in the sense of distributions, of smooth functions
$\psi_{(j_1,i)}$ supported on the set where $1/2 \leq |x_1| \leq 4$
and $1/2 \leq \rho(x) \leq 4$, satisfying the cancellation
conditions (\ref{x_1_cancellation}) and (\ref{x_cancellation})
identically for every ${(j_1,i)}\in \ZZ^2$, and with uniformly
bounded $C^k$ norms for every $k\in \NN$.

Let
$$
\widehat K(\xi_1,\xi) = \sum_{(j_1,i)\in \ZZ^2}
\widehat\psi_{(j_1,i)}(2^{j_1}\xi_1,2^{i}\circ \xi)
$$
be the corresponding product multiplier.

Consider a smooth function $\zeta$ on the real line, supported on
the interval $[1,4]$ and such that $\sum_{k\in \ZZ} \zeta (2^k t)
=1$ for every $t>0$.
For $J=(j_1,j)\in \ZZ^2$, define
$$
\mu_J(\xi_1,\xi) := \sum_{i \in \ZZ}\widehat \psi_{(j_1,i)} (\xi_1,
2^{i-j}\circ \xi)\zeta(\rho(\xi)).
$$
It can be easily proved that the $\mu_J$ form a bounded set of
$\Sch(\RR^3)$. In addition,
a direct computation shows that
$$
\widehat K(\xi_1,\xi)=\sum_{J\in \ZZ^2}  \mu_J(2^{j_1}\xi_1,2^j\circ
\xi)
$$
in the sense of distributions.
Setting \begin{align*}
\varphi_J(x_1,x)&:= ({\mathcal{F}}^{-1}\mu_J)(x_1,x) \\
&\;={\mathcal{F}}_2^{-1}\left(\sum_{i\in
\ZZ}(\mathcal{F}_2\psi_{(j_1,j)})(x_1,
2^{i-j}\circ\cdot)\zeta(\rho(\cdot))\right)(x),
\end{align*}
it is possible to write
 the product kernel $K$ as the sum
\begin{equation}\label{primadecomposizione}
K(x_1,x)  = \sum_{J\in \ZZ^2}
2^{-j_1-jQ}\varphi_J(2^{-j_1}x_1,2^{-j}\circ x),
\end{equation}
convergent in the sense of distributions, of functions $\varphi_J$
that form a bounded set of $\Sch(\RR^3)$ and have compact
$x_1$-support where $1/2\leq |x_1|\leq 4$. 
Finally, 
the fact that $\mu_J(0,\xi)=0$ and $(\partial_\xi^\beta \mu_J)(\xi_1,0)=0$
for every multi-index $\beta=(\beta_1,\beta_2) \in \NN^2$,
identically for every $J\in \ZZ^2$, yields
(\ref{x_1_moment}) for $m=0$ and (\ref{x_moments}).

In fact, we can choose 
$\varphi_J$ so that a finite number of moments
in the variable $x_1$ 
vanish. 
This follows from a slight modification of the arguments in Lemma 2.2.3
in \cite{NagelRicciStein}.
More explicitly, 
denote by $\varphi$ each function $\varphi_J$ in the decomposition
(\ref{primadecomposizione}).
Then each function 
$\varphi$ may be written as a series
\begin{equation}\label{secondadecomposizione}
\varphi(x_1,x)  = \sum_{k\in \ZZ}
2^{-k}A_{k}(2^{-k}x_1,x)\,,
\end{equation} convergent
 in the sense of distributions, of functions $A_k$
which form a bounded set of $\Sch(\RR^3)$ with norms
that decay exponentially in $k$ as $k\to\pm\infty$,
 have compact
$x_1$-support on the set $\{x_1\in\RR\,:\,1/2\leq |x_1|\leq 4\}$,
and satisfy (\ref{x_1_moment})
with $\ell=1$.
 
To prove this fact, consider a  function $\eta\in\ciC^{\infty}_0(\RR)$, supported on
the set $[-4,-1]\cup [1,4]$, such that $\sum_{k\in \ZZ} \eta (2^k t)
=1$ for every $t\neq 0$ and $\int t \,\eta (t) \,dt\neq 0$ .
Set 
$$\chi_k (x_1):=
\eta (2^{-k}x_1)\,,$$
$$\tilde{\chi}_k
(x_1):=\frac{\chi_k (x_1)}{\int x_1 \chi_k (x_1)\,dx_1}\,,$$
$$
a_{k}(x)=
\int x_1\, \chi_k (x_1)\,\varphi (x_1,x)\,dx_1\,,$$
$$
S_k (x)=
 \sum_{j\ge k}
a_j (x)\,.$$
Then write 
$\varphi$ as
 \begin{align*}
 \varphi (x_1,x)&=
 \sum_{k\in\ZZ}
 \big(
 \varphi (x_1,x)
 \chi_k (x_1)-a_k (x)
 \tilde{\chi}_k (x_1)
\big) +
 \sum_{k\in\ZZ}
 \big(
S_k (x)-S_{k+1}(x))
 \tilde{\chi}_k (x_1
 \big)\\
 &=
 \sum_{k\in\ZZ}
 \big(
 \varphi (x_1,x)
 \chi_k (x_1)-a_k (x)
 \tilde{\chi}_k (x_1)
\big) +
 \sum_{k\in\ZZ}
S_k (x)\big(
 \tilde{\chi}_{k} (x_1)-
 \tilde{\chi}_{k-1} (x_1)
 \big)
 \\
 &=
 \sum_{k\in\ZZ}
 A_k (x_1, x)\,,
   \end{align*}
where the series converges in the  sense of distributions,  the functions $A_k$ satisy the moment conditions
$\int x_1 \,A_k (x_1,x)\,dx_1=0$ for all $k\in\ZZ$,
and the Schwartz norms
decay exponentially in $k$ as $k\to\pm\infty$.
Now by rescaling $x_1$ we obtain
(\ref{secondadecomposizione}).
Iterating this  argument yields
 (\ref{x_1_moment})
for all $\ell\le M_1$, for some fixed $M_1\in\NN$.
   
{}

\end{proof}

A result analogous to Lemma \ref{dyadic_decomposition} can be
stated by interchanging the role of $x_1$ and $x$.

\subsubsection*{Estimates on oscillatory integrals}
In the 
following, we prove some estimates on certain  oscillatory integrals related to our problem.
Let $|||\cdot|||$ 
denote any homogeneous norm with respect to the family of
non-isotropic dilations
\begin{equation}  \label{dilations_in R^3}
\delta \bullet (\xi_1,\xi) = (\delta \xi_1, \delta^m \xi_2, \delta^n
\xi_3), \qquad \delta >0\,,
\end{equation}
$e.g.$ we may choose 
 \begin{equation}  \label{homogeneous_norm_in the_space}
|||(\xi_1,\xi)||| =
\max\{|\xi_1|,|\xi_2|^\frac{1}{m},|\xi_3|^\frac{1}{n}\}\,.
\end{equation}
We observe in passing that
\begin{equation}\label{2mn}
\delta \bullet (\xi_1,\xi) = (\delta \xi_1, 
\delta^{2mn}\circ \xi)\,.
\end{equation}

Consider the integral
\begin{equation}\label{integraleoscillanteI}
I^{}(\xi_1,\xi,\eta):=\displaystyle \int
f(x_1,\eta)
e^{-i(\xi_1,\xi)\cdot(x_1,\gamma(x_1))}\,dx_1\,,
\end{equation}
where $\gamma(x_1)=(x_1^m,x_1^n)$
and $f$ is such that
\begin{itemize}
\item[(h1)]
$f$ belongs to $\Sch (\RR^3)$ and is
$x_1$-compactly supported 
on the interval  $1/2 \leq |x_1| \leq 4$;
\item[(h2)] $ f(x_1, 0)=0$ for all $x_1$ such that $ 1/2 \leq |x_1| \leq 4$.
\end{itemize}

The constant $C_N$ occurring in the following
inequalities depend on the Schwartz norms of $f$.
\begin{lemma}\label{oscillatory integrals_1}
Under the hypotheses {\rm{(h1)}} and {\rm{(h2)}}
the following estimate holds for  the integral $I(\xi_1,\xi,\eta)$ defined  by (\ref{integraleoscillanteI})
\begin{equation} \label{estimate_1}
| \, I(\xi_1,\xi,\eta)| \leq C_N
\frac{\rho(\eta)^{\frac{1}{2n}}}{(1+\rho(\eta))^N}
\end{equation}
for every integer $N\geq0$.
\end{lemma}
\begin{proof} 
Since $f$ is a Schwartz function,
by using  {\rm{(h2)}}, (\ref{stimerho}) and (\ref{stimerho1}) 
 we deduce that for every 
$(x_1,\eta) \in \{x_1\in
\RR:\, 1/2 \leq |x_1|\leq 4\} \times \RR^2$
\begin{equation} \label{Schwartz_estimates}
|  \partial_{x_1}^k  f(x_1,\eta)| \leq
C_N \frac{\rho(\eta)^{\frac{1}{2n}}}{(1+\rho(\eta))^N}
\end{equation}
for all $N \in \NN$ and $k\in\NN$.
For $k=0$ 
this inequality 
yields then (\ref{estimate_1}).
\end{proof}

\begin{lemma}\label{oscillatory integrals_2}
Let $I$ be the oscillatory integral defined by
(\ref{integraleoscillanteI}).
Assume that
 {\rm{(h1)}} and {\rm{(h2)}}
are satisfied.

Then 
\begin{equation} \label{estimate_3}
|I^{}(\xi_1,\xi,\eta)| \leq C_N
\frac{\rho(\eta)^{\frac{1}{2n}}}{(1+\rho(\eta))^N|||(\xi_1,\xi)|||^\frac{1}{n}}
\qquad \mbox{if} \quad |||(\xi_1,\xi)|||>1
\end{equation}
for every integer $N\geq0$.
\vskip0.2cm
Under the additional assumption
\begin{equation} \label{x_1_cancellation_property}
\displaystyle \int f(x_1,x)\,dx_1=0\,,\;x\in\RR^2\,,\\
\end{equation}
the following estimate holds
\begin{equation} \label{estimate_2}
|I^{}(\xi_1,\xi,\eta)| \leq C_N
\frac{\rho(\eta)^{\frac{1}{2n}}}{(1+\rho(\eta))^N} |||(\xi_1,\xi)||| \qquad
\mbox{if} \quad |||(\xi_1,\xi)|||\leq 1
\end{equation}
for every integer $N\geq0$.
\end{lemma}
\begin{proof}
Let $|||(\xi_1,\xi)|||$ be the homogeneous norm defined in
(\ref{homogeneous_norm_in the_space}). We divide the phase space
$\xi_1,\xi$ into two regions, depending on whether
$|||(\xi_1,\xi)||| \leq 1$ or $|||(\xi_1,\xi)||| >1$.

Assume that $|||(\xi_1,\xi)||| >1$
and
write $(\xi_1,\xi) = \lambda \omega, $ where $\lambda = |||(\xi_1,\xi)|||$
and $\omega={(\xi_1,\xi)}/{|||(\xi_1,\xi)|||}$.

With these notation, the oscillatory integral
$I^{}(\xi_1,\xi,\eta)$ becomes
\begin{equation}  \label{rewriting_the_oscillatory_integral}
I^{}(\lambda\omega,\eta)=\displaystyle \int_{\frac12
\leq |x_1|\leq 4} 
f(x_1,\eta)
e^{-i\lambda\omega\cdot(x_1,\gamma(x_1))}\,dx_1\,.
\end{equation}
Set
\begin{equation}\label{fase}
\Phi(x_1) = \omega \cdot (x_1,\gamma(x_1))=
\frac{1}{|||(\xi_1,\xi)|||}(\xi_1x_1+\xi_2x_1^m+\xi_3x_1^n),
\end{equation}
for $1/2 \leq |x_1| \leq 4$.

We observe that, since
the curve $x_1\mapsto (x_1,x_1^m,x_1^n)$ for $1/2 \leq |x_1|
\leq 4$ is of finite type n,  for every $x_1\in \RR$
with $1/2 \leq |x_1| \leq 4$ there exists a positive integer
$\overline{n}\leq n$ such that $\Phi^{(\overline{n})}(x_1)\neq 0$
\cite{Stein}. 

A standard application of  the
Van der Corput's lemma to the oscillatory integral (\ref{rewriting_the_oscillatory_integral}),
the compactness of the $x_1$-support , and 
\begin{equation}  \label{estimate_on derivatives_of_phase_function}
|\Phi^{({\overline{n})}}(x_1)| \geq C
\end{equation}
for some $C>0$,
for all $x_1$ such that  $\frac14\le |x_1|\le  4$, yield
\begin{equation}\label{Van_der_Corput_estimate}
|I^{}(\lambda\omega,\eta)| \leq  C
\lambda^{-\frac{1}{n}} \|\partial_{x_1}
f(x_1,\eta)\|_{L^1(\RR_{x_1})}\,.
\end{equation}

Since the function $f(x_1,\eta)$ is smooth and compactly supported
in $x_1$ where $\frac12\leq |x_1|\leq4$, by using the estimate
(\ref{Schwartz_estimates}) in the case $k=1$
 we see that
$$
\|\partial_{x_1}f(x_1,\eta)\|_{L^1(\RR_{x_1})} \leq C_N
\frac{\rho(\eta)^{\frac{1}{2n}}}{(1+\rho(\eta))^N}
$$
for every $N\in \NN$.
As a consequence,
$$
|I^{}(\lambda\omega,\eta)| \leq  C_N
\lambda^{-\frac{1}{n}} \frac{\rho(\eta)^{\frac{1}{2n}}}{(1+\rho(\eta))^N}
$$
for every $N\in \NN$.
Since $\lambda=|||(\xi_1,\xi)|||$, we get (\ref{estimate_3}).

Assume now that $|||(\xi_1,\xi)|||\leq 1$. The cancellation condition
(\ref{x_1_cancellation_property}), the Mean value theorem and the
estimate (\ref{Schwartz_estimates}) 
can be used to prove
that
$$
\begin{aligned}
|I^{}(\xi_1,\xi,\eta)| &\leq \int_{\frac12 \leq
|x_1|\leq 4} 
|f(x_1,\eta)|
|e^{-i(\xi_1,\xi)\cdot(x_1,\gamma(x_1))}-1|\,dx_1\\
&\leq C_N \frac{\rho(\eta)^{\frac{1}{2n}}}{(1+\rho(\eta))^N} |(\xi_1,\xi)|,
\end{aligned}
$$
for every $N\in \NN$.
Since 
by hypothesis $|||(\xi_1,\xi)|||\leq 1$, we have that
$|(\xi_1,\xi)| \leq 3|||(\xi_1,\xi)|||$. This inequality, together
with the previous estimate, yields (\ref{estimate_2}).

\end{proof}

The estimate (\ref{estimate_3}) in Lemma \ref {oscillatory integrals_2}
can be improved in the region of the
space $(\xi_1,\xi)$ where the first derivative of the phase
(\ref{fase})
never vanishes for $1/2\leq
|x_1|\leq4$,
as the following lemma shows.

\begin{lemma}\label{oscillatory integrals_3}

Let $I$ be the oscillatory integral defined by
(\ref{integraleoscillanteI}).
Under the hypotheses {\rm{(h1)}} and {\rm{(h2)}} there exists a costant $\tilde{C}>1$
such that
for every integer $N\geq0$ 
\begin{equation} \label{estimate_4}
|I^{}(\xi_1,\xi,\eta)| \leq C_N
\frac{\rho(\eta)^{\frac{1}{2n}}}{(1+\rho(\eta))^N|||(\xi_1,\xi)|||^N}
\qquad \mbox{when} \quad |||(\xi_1,\xi)|||>1\,
\end{equation}
and
$$
|\xi_1| > \tilde{C}(|\xi_2|+|\xi_3|)\,,  \qquad \mbox{or} \qquad |\xi_2|
> \tilde{C}(|\xi_1|+|\xi_3|) \,, \qquad \mbox{or} \qquad |\xi_3| >
\tilde{C}(|\xi_1|+|\xi_2|)
\,.$$
\end{lemma}

\begin{proof}
We  use for the integral $ I$ the notation introduced in formula
($\ref{rewriting_the_oscillatory_integral}$).

In order to improve the estimate
(\ref{estimate_3}), we have to determine the subsets of
the phase space $(\xi_1, \xi)$
where
\begin{equation}  \label{zeros_of_the
first_derivative_of_the_phase}\Phi'(x_1)=
{\xi_1}+m{\xi_2}x_1^{m-1} +n{\xi_3}x_1^{n-1}=
0\,
\end{equation}
for some
$x_1\in\RR$ such that
$\frac12\le |x_1|\le 4$.
Some elementary estimates show that
 we can find a constant $\tilde{C}>1$ sufficiently large,
such that for any fixed point $(\xi_1,\xi)\in \RR^3$, satisfying
$|||(\xi_1,\xi)|||>1$ and
$$
|\xi_1| > \tilde{C}(|\xi_2|+|\xi_3|)
\qquad \mbox{or} \qquad |\xi_2|
> \tilde{C}(|\xi_1|+|\xi_3|)  \qquad \mbox{or} \qquad |\xi_3| >
\tilde{C}(|\xi_1|+|\xi_2|),
$$
there
exists a constant $C_\omega>0$ such that
\begin{equation}  \label{estimate_on_the_first_derivative_of_the_phase}
|\Phi'(x_1)| \geq C_\omega
\end{equation}
for every $x_1\in \RR$ with $1/2 \leq |x_1| \leq 4$.

Let $D$ denote the differential operator
$$
Df(x_1,\eta) = (-i\lambda \Phi'(x_1))^{-1}\frac{\partial f}{\partial
x_1}(x_1,\eta)
$$
and let $^tD$ denote its transpose
$$
^tDf(x_1,\eta) = \frac{\partial}{\partial
x_1}\left(\frac{f}{i\lambda\Phi'(x_1)}\right).
$$
 Since $D^N(e^{-i\lambda\Phi(x_1)})
= e^{-i\lambda\Phi(x_1)}$ for every $N\in \NN$, integration by parts
shows that
$$
\begin{aligned}
I^{}(\lambda\omega,\eta)&=\displaystyle \int_{\frac12
\leq |x_1|\leq 4} f(x_1,\eta)
D^N\left(e^{-i\lambda\Phi(x_1)}\right)\,dx_1 \\
&= \displaystyle \int_{\frac12 \leq |x_1|\leq 4} (^tD)^N
f(x_1,\eta)e^{-i\lambda\Phi(x_1)}\,dx_1.
\end{aligned}
$$
Since $f(x_1,\eta)$ is a smooth function with compact support in the
$x_1$ variable in the region where $1/2 \leq |x_1| \leq 4$, $f$ satisfies  the
estimate (\ref{Schwartz_estimates}), and $\Phi(x_1)$ is a smooth
function satisfying the inequality
(\ref{estimate_on_the_first_derivative_of_the_phase}), we can verify
that
$$
|(^tD)^N f(x_1,\eta)| \leq C_{N, \omega}
\frac{\rho(\eta)^{\frac{1}{2n}}}{(1+\rho(\eta))^N} \lambda^{-N}
$$
for every $N \in\NN$.

Therefore we conclude that
$$
|I^{}(\lambda\omega,\eta)| \leq C_{N, \omega}\,\lambda^{-N} \frac{\rho(\eta)^{\frac{1}{2n}}}{(1+\rho(\eta))^N}
$$
for every $N\in \NN$.
By a compactness argument  we can show that the previous estimate
is independent of $\omega$, so that
$$
|I^{}(\lambda\omega,\eta)| \leq  C_N \lambda^{-N}
\frac{\rho(\eta)^{\frac{1}{2n}}}{(1+\rho(\eta))^N}
$$
for every $N\in \NN$.
Since $\lambda=|||(\xi_1,\xi)|||$, we obtain the inequality
(\ref{estimate_4}).
\end{proof}

\begin{remark}\label{remark_1}
In the sequel, we shall sistematically apply the estimates
(\ref{estimate_1}),
(\ref{estimate_3}),
(\ref{estimate_2}),
and
(\ref{estimate_4})
to the oscillatory integral  (\ref{integraleoscillanteI})
with the integrand $f(x_1,\eta)$  of the  form
$
x_1^{\alpha} ({\mathcal F}_2(x^\beta\varphi_J))(x_1,\eta)$, where the functions
 $\varphi_J$ are given by Lemma
\ref{dyadic_decomposition}.
In particular, 
the Schwartz norms $\|\cdot\|_{(N)}$ 
of the functions $\varphi_J$ are uniformly bounded in $J\in \ZZ^2$ for every $N\in \NN$.

We observe that  the functions
$x_1^{\alpha}  ({\mathcal F}_2(x^\beta\varphi_J))(x_1,\eta)$
fullfil the hypotheses (h1) and (h2), as a consequence of the 
cancellation condition (\ref{x_moments}).

Moreover  we have
$$\int x_1^{\ell} ({\mathcal F}_2(x^\beta\varphi_J))(x_1,\eta)dx_1=0
$$
as a consequence of the cancellation property (\ref{x_1_moment})
for all $\ell \le M_1$ for some fixed $M_1\in\NN$, so that  (\ref{x_1_cancellation_property}) 
is satisfied.
\end{remark}
\medskip

\section{{$L^2$}-boundedness}

Let $K$ be the kernel defined by (\ref{adapted_product_kernel_def})
and $T$ the operator given by $T:f\mapsto f
*K$.
In this section we  prove that $T$ is bounded on
$L^2(\RR^3)$.

Let $J=(j_1,j)$. We proved in Lemma \ref{dyadic_decomposition} that the
product kernel $K_0$ can be written as a sum
$$
K_0(x_1,x) = \sum_{J\in \ZZ^2} 2^{-j_1-jQ}
\varphi_J(2^{-j_1}x_1,2^{-j}\circ x)
$$
convergent in the sense of distributions, of Schwartz functions
$\{\varphi_J\}_{J\in \ZZ^2}$ on $\RR^3$, satisfying the properties
(i), (ii) and (iii) of Lemma \ref{dyadic_decomposition}.

\begin{proposition}\label{L^2_boundedness}
The series 
\begin{equation}  \label{dyadic_decomp_for_K}
K(x_1,x)= \sum_{J\in \ZZ^2} 
2^{-j_1-jQ}
\varphi_J\left(2^{-j_1}x_1,2^{-j}\circ (x-\gamma(x_1)\right)\,
\end{equation}
 converges in the sense of
distributions and the corresponding convolution operator $T$ is
bounded on $L^2(\RR^3)$.
\end{proposition}

\begin{proof}
Let $m_J(\xi_1,\xi)$ be the Fourier transform of the $J$-th summand
of the series (\ref{dyadic_decomp_for_K}). With  a change of
variables, we may write
\begin{equation}  \label{Fourier_transform}
\begin{aligned}
m_J(\xi_1,\xi) &= \displaystyle \int_{1/2\leq |x_1|\leq 4}\int
\varphi_J(x_1,x)
e^{-i(\xi_1,\xi)\cdot(2^{j_1}x_1, 2^j\circ x+\gamma(2^{j_1}x_1))}\,dx_1dx\\
&=\displaystyle \int_{1/2\leq |x_1|\leq 4}({\mathcal
F}_2\varphi_J)(x_1,2^j\circ
\xi)e^{-i(2^{j_1}\xi_1,2^{2mn\,j_1}\circ\xi)\cdot(x_1,
\gamma(x_1))}\,dx_1.
\end{aligned}
\end{equation}
Let $|||(\xi_1,\xi)|||$ be any  norm homogeneous
 with respect to the family of
non-isotropic dilations
(\ref{dilations_in R^3}),
$e.g.$ we may choose the norm defined by
(\ref{homogeneous_norm_in the_space}). We decompose the series
\begin{equation}  \label{adapted_product_multiplier}
\sum_{J\in \ZZ^2} m_J(\xi_1,\xi)
\end{equation}
as
\begin{equation}  \label{splitted_product_multiplier}
\sum_{\tiny{\begin{array}{ll}
2^{j_1}|||(\xi_1,\xi)||| \leq 1\\
j\in \ZZ \end{array}}} m_J(\xi_1,\xi)+ \sum_{\tiny{\begin{array}{ll}
2^{j_1}|||(\xi_1,\xi)||| > 1\\
j\in \ZZ \end{array}}} m_J(\xi_1,\xi)\,.
\end{equation}
It sufficies to prove
that each of the two series in 
(\ref{splitted_product_multiplier}) converges boundedly to a bounded
function.
To do this, 
we  apply Lemma \ref{oscillatory integrals_2}
to the oscillatory integral
\begin{equation}\label{oscillatory_int_I_J}
I_J^{}(\xi_1,\xi,\eta):=\displaystyle \int_{\frac12
\leq |x_1|\leq 4} 
(\mathcal{F}_2(\varphi_J))(x_1,\eta)
e^{-i(\xi_1,\xi)\cdot(x_1,\gamma(x_1))}\,dx_1\,.
\end{equation} 
The functions
$(\mathcal{F}_2(\varphi_J))(x_1,\eta)$ have  Schwartz norms $\|\varphi_J \|_{(N)}$ 
uniformly bounded in $J\in \ZZ^2$ for every $N\in \NN$.
As observed in
Remark \ref{remark_1},
they
satisfy 
(h1), (h2) and
(\ref{x_1_cancellation_property}), so that both estimate
(\ref{estimate_3}) and
(\ref{estimate_2})
hold, with costants $C_N$ independent of 
$N$.

More precisely, consider the first series in (\ref{splitted_product_multiplier}). Since
\begin{equation} \label{link_between_multiplier_and_I_J}
m_J(\xi_1,\xi) =
I_J^{}(2^{j_1}\xi_1,2^{2mn\,j_1}\circ\xi,2^{j}\circ\xi),
\end{equation}
by applying
(\ref{estimate_2}) we obtain
\begin{align*}
\sum_{\tiny{\begin{array}{ll}
2^{j_1}|||(\xi_1,\xi)||| \leq 1\\
j\in \ZZ \end{array}}} |m_J(\xi_1,\xi)| &\leq C_N
\sum_{\tiny{\begin{array}{ll}
2^{j_1}|||(\xi_1,\xi)||| \leq 1\\
j\in \ZZ \end{array}}} \frac{\rho(2^j\circ
\xi)^{\frac{1}{2n}}}{(1+\rho(2^j\circ\xi))^N}
|||(2^{j_1}\xi_1,2^{2mn\,j_1}\circ\xi)|||\\
&=\sum_{\tiny{\begin{array}{ll}
2^{j_1}|||(\xi_1,\xi)||| \leq 1\\
j\in \ZZ \end{array}}} \frac{(2^j\rho(
\xi))^{\frac{1}{2n}}}{(1+2^j\rho(\xi))^N}
2^{j_1}|||(\xi_1,\xi)|||\\
&\leq C_N \sum_{j\in
\ZZ}\frac{(2^j\rho(\xi))^{\frac{1}{2n}}}{(1+2^j\rho(\xi))^N},
\end{align*}
for every $N\in \NN$. Since the series
$$
\sum_{j\in \ZZ}\frac{(2^jb)^{\frac{1}{2n}}}{(1+2^jb)^N}
$$
is uniformly bounded in $b$, it follows that the
series in the previous formula converges boundedly to a bounded
function.

By using the identity (\ref{link_between_multiplier_and_I_J}) and
the inequality (\ref{estimate_3}) we prove that also the second sum
on the righthand side of (\ref{splitted_product_multiplier})
converges boundedly to a bounded function.

This proves that the series in (\ref{adapted_product_multiplier})
converges boundedly (and hence in the sense of distributions) to a
bounded function $m(\xi_1,\xi)$. As a consequence the series
(\ref{dyadic_decomp_for_K}) converges in the sense of distributions
to the distribution $K=\mathcal{F}^{-1}(m)$.

Finally, by Plancherel's theorem, the boundedness of $m$ implies
that the corresponding operator $T$ is bounded on $L^2(\RR^3)$.
\end{proof}

\section{{$L^p$}-boundedness}
In this section we 
prove the $L^p$-boundedness of the operator $T$.
For this we 
split the sum (\ref{dyadic_decomp_for_K}) into two
parts
\begin{equation}
\begin{aligned}
K(x_1,x) &=\sum_{2mn j_1\leq j}2^{-j_1-jQ}
\varphi_J\left(2^{-j_1}x_1,2^{-j}\circ (x-\gamma(x_1)\right)\\
&\qquad +\sum_{2mn  j_1> j}2^{-j_1-jQ}
\varphi_J\left(2^{-j_1}x_1,2^{-j}\circ (x-\gamma(x_1)\right)\\
&=:K_1(x_1,x)+K_2(x_1,x).
\end{aligned}
\end{equation}
Correspondingly we break
the operator $T$ into the  sum
$$
Tf = f * K_1 + f* K_2 =:T_1f +T_2f
$$
and we prove that  $T_1$ and $T_2$ are bounded on $L^p(\RR^3)$
for $1<p<\infty$.

\begin{proposition}\label{L^p_boundedness_of_T_1}
The operator $T_1$ is bounded on $L^p(\RR^3)$ for $1<p<\infty$.
\end{proposition}

\begin{proof}
Let $m_J(\xi_1,\xi)$ be the multiplier given in
(\ref{Fourier_transform}). We show that the series
$$
\widehat{K_1}(\xi_1,\xi) = \sum_{2mn j_1\leq j} m_J(\xi_1,\xi)
$$
defines a Marcinkiewicz multiplier on $\RR\times\RR^2$ adapted to
the dilations (\ref{dilations_in_R^2}) on $\RR^2$. As a consequence,
we will obtain  that $T_1$ is bounded on $L^p(\RR^3)$.
This is part of the folklore, for a formal proof see \cite{R}.

It suffices to show
that $\widehat{K_1}(\xi_1,\xi)$ is a bounded function on $\RR^3$
such that for each $s_1\in\{0,1\}$ and for each
multi-index $s = (s_2,s_3)
\in \NN^2$ with $|s| \leq 2$ 
there is a positive constant
$C_{s_1,s}$ 
for which 
\begin{equation} \label{diff_ineq}
|\partial_{\xi_1}^{s_1}
\partial_\xi^{s}\widehat{K_1}(\xi_1,\xi)|\leq
C_{s_1,s}
|\xi_1|^{-s_1}
\rho(\xi)^{-\frac{s_2}{2n}-\frac{s_3}{2m}}
\end{equation}
for every $(\xi_1,\xi)\in \RR\times\RR^2$ with $\xi_1 \neq 0$ and
$\xi \neq 0$.

We already proved in
Proposition \ref{L^2_boundedness} that
$\widehat K_1(\xi_1,\xi)$ is a bounded
function on $\RR^3$. 

We give the proof of the differential inequalities (\ref{diff_ineq})
for $s_1=1$ and $s=(0,0)$ and for
$s_1=0$ and $s=(1,0)$, the other cases
being essentially the same, with the extra disadvantages of more
complicated notation and computations.

Set 
\begin{equation}\label{oscillatory_int_I_Jab}
I_J^{\alpha,\beta}(\xi_1,\xi,\eta):=\displaystyle \int_{\frac12
\leq |x_1|\leq 4} x_1^{\alpha}
(\mathcal{F}_2(x^{\beta}\varphi_J))(x_1,\eta)
e^{-i(\xi_1,\xi)\cdot(x_1,\gamma(x_1))}\,dx_1\,,
\end{equation} 
where
$\alpha\in \NN$, $\alpha \leq 2$ and
$\beta=(\beta_1,\beta_2)\in \NN^2$ with $|\beta| \leq 2$.

We first consider the case  $s_1=1$ and $s =
(0,0)$. Since
$$
\partial_{\xi_1}m_J(\xi_1,\xi) = -i2^{j_1}
I_J^{1,(0,0)}(2^{j_1}\xi_1,2^{2mn\,j_1}\circ\xi,2^{j}\circ\xi)\,,
$$
we write
\begin{align*}
\sum_{2mn\,j_1 \leq j}|\partial_{\xi_1}m_J(\xi_1,\xi)| &\leq
|\xi_1|^{-1}\biggl( \sum_{\tiny{\begin{array}{ll}
2^{j_1}|||(\xi_1,\xi)||| \leq 1\\
j\in \ZZ \end{array}}}
2^{j_1}|\xi_1|\,|I_J^{1,(0,0)}(2^{j_1}\xi_1,2^{2mn\,j_1}\circ\xi,2^{j}\circ\xi)|\\
&\qquad+\sum_{\tiny{\begin{array}{ll}
2^{j_1}|||(\xi_1,\xi)||| > 1\\
2mn\,j_1\leq j\\
2^{j_1}|\xi_1| \leq {\tilde{C}}(2^{mj_1}|\xi_2|+2^{nj_1}|\xi_3|)
\end{array}}}
2^{j_1}|\xi_1|\,|I_J^{1,(0,0)}(2^{j_1}\xi_1,2^{2mn\,j_1}\circ\xi,2^{j}\circ\xi)|\\
&\qquad+\sum_{\tiny{\begin{array}{ll}
2^{j_1}|||(\xi_1,\xi)||| > 1\\
2mn\,j_1\leq j\\
2^{j_1}|\xi_1| > {\tilde{C}}(2^{mj_1}|\xi_2|+2^{nj_1}|\xi_3|)
\end{array}}}
2^{j_1}|\xi_1|\,|I_J^{1,(0,0)}(2^{j_1}\xi_1,2^{2mn\,j_1}\circ\xi,2^{j}\circ\xi)|\biggr)\\
&=:\Sigma_1+\Sigma_2+\Sigma_3\,.
\end{align*}
Since
$2^{j_1}|\xi_1|\leq
2^{j_1}|||(\xi_1,\xi)|||$, the convergence of $\Sigma_1$ follows from
the estimate
(\ref{estimate_1})
applied to the integral
$I_J^{1,(0,0)}(2^{j_1}\xi_1,2^{2mn\,j_1}\circ\xi,2^{j}\circ\xi)$.
 
The inequalities $2^{j_1}|\xi_1| \leq
{\tilde{C}}(2^{mj_1}|\xi_2|+2^{nj_1}|\xi_3|)$ and $2mn\,j_1\leq j$ imply
that $2^{j_1}|\xi_1| \leq
{\tilde{C}}(2^{\frac{j}{2n}}|\xi_2|+2^{\frac{j}{2m}}|\xi_3|) \leq
{\tilde{C}}((2^j\rho(\xi))^{\frac{1}{2n}}+(2^j\rho(\xi))^{\frac{1}{2m}})$. This fact,
together with the estimate (\ref{estimate_3}) for the integral
$I_J^{1,(0,0)}(2^{j_1}\xi_1,2^{2mn\,j_1}\circ\xi,2^{j}\circ\xi)$ shows
that also $\Sigma_2$  converge to a bounded function.
Finally, 
the sum $\Sigma_3$ converges because of 
(\ref{estimate_4}).
Therefore (\ref{diff_ineq}) holds for $s_1=1$ and
$s = (0,0)$.

Now, assume that $s_1=0$ and
$s=(1,0)$, then we have that
$$
\partial_{\xi_2}m_J(\xi_1,\xi) = -i2^{j/2n}
I_J^{0,(1,0)}(2^{j_1}\xi_1,2^{2mn\,j_1}\circ\xi,2^{j}\circ\xi)
-i2^{mj_1}
I_J^{m,(0,0)}(2^{j_1}\xi_1,2^{2mn\,j_1}\circ\xi,2^{j}\circ\xi).
$$
We write
\begin{equation}  \label{xi_2_partial_derivative}
\begin{aligned}
\sum_{2mn\,j_1 \leq j}|\partial_{\xi_2}m_J(\xi_1,\xi)| &\leq
\rho(\xi)^{-{\frac{1}{2n}}}\biggl( \sum_{\tiny{\begin{array}{ll}
2^{j_1}|||(\xi_1,\xi)||| \leq 1\\
j\in \ZZ \end{array}}}
(2^j\rho(\xi))^{\frac{1}{2n}}\,|I_J^{0,(1,0)}(2^{j_1}\xi_1,2^{2mn\,j_1}\circ\xi,2^{j}\circ\xi)|\\
&\qquad+\sum_{\tiny{\begin{array}{ll}
2^{j_1}|||(\xi_1,\xi)||| > 1\\
j\in \ZZ
\end{array}}}
(2^j\rho(\xi))^{\frac{1}{2n}}\,|I_J^{0,(1,0)}(2^{j_1}\xi_1,2^{2mn\,j_1}\circ\xi,2^{j}\circ\xi)|\\
&\qquad+\sum_{\tiny{\begin{array}{ll}
2^{j_1}|||(\xi_1,\xi)||| \leq 1\\
j\in \ZZ
\end{array}}}
(2^j\rho(\xi))^{\frac{1}{2n}}\,|I_J^{m,(0,0)}(2^{j_1}\xi_1,2^{2mn\,j_1}\circ\xi,2^{j}\circ\xi)|\biggr)\\
&\qquad+\sum_{\tiny{\begin{array}{ll}
2^{j_1}|||(\xi_1,\xi)||| > 1\\
j\in \ZZ
\end{array}}}
(2^j\rho(\xi))^{\frac{1}{2n}}\,|I_J^{m,(0,0)}(2^{j_1}\xi_1,2^{2mn\,j_1}\circ\xi,2^{j}\circ\xi)|\biggr).
\end{aligned}
\end{equation}
By using the estimate (\ref{estimate_2}) we can easily prove that
the first and the third series 
on the righthand side of
(\ref{xi_2_partial_derivative}) converge to a bounded function.
Also, the second and the fourth series on the righthand side of
(\ref{xi_2_partial_derivative}) converge to a bounded function as we
can see by applying the estimate (\ref{estimate_3}). Hence
(\ref{diff_ineq}) holds for $s_1=0$ and
$s = (1,0)$.
\end{proof}

We  now prove  the $L^p$-boundedness of the operator $T_2$
by means of  the analytic  interpolation method.
We start constructing an  analytic family of linear operators
$T_{2,z}$.

For $z\in \CC$ we consider the kernel  and its analytic
continuation, defined in  (\ref{IMN}), 
$$
I^z(u) = \frac{G(0)\rho(u)^{z-Q}}{\sigma(S)\Gamma(z)G(z)}
$$
on $\RR^2$, where $u=(u_2,u_3)$, $\sigma(S)$ denotes the surface
measure of the sphere $$S:=\{u\in\RR^2\,:\, \rho (u)=1 \}\,,$$ and
$G(z)$ has been defined in
(\ref{GMN}).

\begin{example}
In the light of Example \ref{ esempio6-4}, when $m=2$ and $n=3$ we
have
\begin{equation}\label{Dduetre}
\begin{aligned}
&G(z):= \left( \Gamma \left( z+\frac{7}{12}\right) \right)^2
\cdot \Gamma\left( {z+{\frac{1}{6}}} \right) \Gamma\left(
{z+{\frac{1}{4}}}\right)
\Gamma\left( {z+\frac{1}{3}} \right)\\
&\quad
\Gamma\left( {z+\frac{5}{12}}
\right)\Gamma\left( {z+\frac{1}{2}}\right) \Gamma\left( {z+\frac{2}{3}}
\right) \Gamma\left( {z+\frac{3}{4}} \right)
\Gamma\left( {z+\frac{5}{6}}\right)
\Gamma\left( {z+\frac{11}{12}} \right) \,.
\end{aligned}\end{equation}\end{example}


 We shall
denote by $B_r$, $r>0$,
 the
non-isotropic ball
\begin{equation}\label{palla}
B_{r}:=\left\{ (\uu,\ud)\in\RR^2 \,:\,\,\uu^{ 2n}+\ud^{2m}\le r
  \right\}\,.
\end{equation}
Let $\theta$ be a smooth compactly supported function on $\RR^2$
whose support is contained in the ball $B_{\frac{1}{4}}$
and which is identically one in a neighbourhood of the origin. Let
$B^z$ be the distribution defined by
\begin{equation}\label{B^z}
\langle B^z,h\rangle:= \langle I^z,\theta h\rangle, \qquad h\in
\Sch(\RR^2).
\end{equation}
Observe  that
$B^0=\delta_0$,
since $I^0 =\delta_0$ and $\theta$ is identically one in a
neighbourhood of the origin.

By considering the convolution between
$2^{-j_1-jQ}\varphi_J(2^{-j_1}x_1,2^{-j}\circ(x-\gamma(x_1))$ and
$2^{-(m+n)j_1}(\delta_0\otimes B^z)(x_1,2^{-2mn\,j_1}\circ x)$ we get the
kernel
$$
K_{2,z}(x_1,x) \!= \!\sum_{2mn\,j_1>j}\!2^{-(m+n+1)j_1-jQ}
\displaystyle\!\int\varphi_J(2^{-j_1}x_1,2^{-j}\circ(u-\gamma(x_1)))
B^z(2^{-2mnj_1}\circ(x-u))du\,.\\
$$
If we set
\begin{equation}\label{lambdaJ}
\lambda_J(x_1,x): =
2^{(2mn\,j_1-j)Q}\displaystyle\int\varphi_J(x_1,2^{2mn\,j_1-j}\circ(x-u-\gamma(x_1)))
B^z(u)\,du\,,
\end{equation}
then  
the kernel $K_{2,z}$
may be written, by a change of variable, as
\begin{equation}\label{K_2z}
K_{2,z}(x_1,x):=\sum_{2mn\,j_1>j}2^{-(m+n+1)j_1}\lambda_J(2^{-j_1}x_1,2^{-2mn\,j_1}\circ x)\,.
\end{equation}
We consider the analytic family of operators
(of admissible growth)
\begin{equation}\label{T2z}
T_{2,z}f:=f\ast K_{2,z}, \qquad f\in \Sch(\RR^3)\,,
\end{equation}
and we observe that $T_{2,0}f=f\ast K_{2,0}=T_2f$.

In order to prove the $L^p$-boundedness of $T_2$
we  need some preliminary results.
The first result is  
an $L^1$- Lipschitz condition
for  the distribution $B^z$ defined by (\ref{B^z}).
\begin{lemma}\label{lemmaBz}
For all $0<\Re e z< Q$
\begin{equation}\label{LipschitzBz}
\int \Big| B^z (u+h)-B^z (u)\Big|\,du \le
C_z \rho (h)^{\Re e z}
\end{equation}
for all $h\in\RR$.
\end{lemma}
\begin{proof}
First we split the integral in
(\ref{LipschitzBz})
in two parts
\begin{align*}
&\int \Big|B^z (u+h)-B^z (u)\Big|\,du=
\int_{
\rho(h)\le \frac{\rho(u)}{2} } 
\Big| B^z (u+h)-B^z (u)\Big|du\\
&\qquad\qquad\quad + \int_{
\rho(h)> \frac{\rho(u)}{2} } 
\Big| B^z (u+h)-B^z (u)\Big|du=:I+II\,.
\\
\end{align*}
Now 
\begin{align*}
&
I\le
\int_{
\rho(h)\le \frac{\rho(u)}{2} } 
 \Big|C_z \theta  (u+h)
\big(   \rho^{z-Q}(u+h)- \rho^{z-Q}(u)\big)
\Big|\,du\\
&\qquad\qquad
\qquad
+\int_{
\rho(h)\le \frac{\rho(u)}{2} } 
\Big| C_z 
\big( \theta  (u+h)-\theta (u)\big)
 \rho^{z-Q}(u)
 \Big|du=:I_a +I_b\,.\\
\end{align*}
To estimate $I_a$, we use the Mean Value Theorem 
[FoS, p. 11], obtaining
\begin{align*}
&
I_a\le
 C
\int_{
\rho(h)\le \frac{\rho(u)}{2} } 
  \big( \rho(u)\big)^{\Re e z-Q-1} \rho(h)
du
= C\rho (h)
\int_ {2\rho (h)}^{+\infty } r^{\Re e z -Q-1}r^{Q-1}dr
  \int_S \rho (v)^{\Re e z -Q-1} d\sigma (v)
\\
&\le C\rho (h)\rho (h)^{\Re e z -1}
=
  C\rho (h)^{\Re  e z }\,.
\end{align*}

To estimate $I_b$, we observe that
\begin{align*}
&
I_b
\le C
\int_{
\rho(h)\le \frac{\rho(u)}{2} } 
 \big| 
 \theta  (u+h)-\theta(u)\big|  \big(\rho(u)\big)^{\Re e z-Q}
\,du\\
&\le C
\rho(h)^{\Re e z-Q}
\int_{
\rho(h)\le \frac{\rho(u)}{2} } 
 \big| 
 \theta  (u+h)-\theta(u)\big| 
\,du
\le 2 C
\rho(h)^{\Re e z-Q}
\int
 \big| 
 \theta  (u)\big| 
\,du\\
&\le C
\rho(h)^{\Re e z-Q}\,.\\
\end{align*}
Finally, if $k$ denotes a positive costant such that
$\rho (x+y)
\le k \big( \rho (x)+\rho (y)\big)$
we observe that
by [FoS, p. 14]
\begin{align*}
&
 II= 
 \int_{
\rho(h)> \frac{\rho(u)}{2} } 
\Big| B^z (u+h)-B^z (u)\Big|du
\\
&\le
2 C
\int_{
\rho(h)> \frac{\rho(u)}{3k} }
\big(\rho (u) \big)^{\Re e z -Q}du=
2 C
\int_{0}^{3k
\rho(h) }
r^{\Re e z -Q}r^{Q-1}dr
  \int_S \rho (v)^{\Re  e z -Q} d\sigma (v)\\
&\le
C
\rho(h)^{\Re e z}\,.
\\
\end{align*}
This inequality, combined with the bounds for $I$,
yields 
(\ref{LipschitzBz}).


\end{proof}

Then we need to recall the definition of  non-isotropic B{e}sov  spaces \cite{Stein}.

\begin{definition}
In $\RR^3$ we consider the family of one-parameter non-isotropic
dilations defined in (\ref{dilations_in R^3}). 
Let 
$
\widetilde
\rho(x_1,x)$ be any
homogeneous norm
with respect to these dilations.

We denote by $B^\alpha_{1,\infty}$ the non-isotropic B{e}sov space
of functions $f\in L^1(\RR^3)$ which satisfy an $L^1$-Lipschitz
condition of order $\alpha$, $0<\alpha<1$, i.e. there exists a
positive constant $C$ such that
$$
\displaystyle\int |f(x_1+h_1,x+h)-f(x_1,x)|\,dx_1dx \leq
C\widetilde\rho(h_1,h)^\alpha
$$
for every $(h_1,h)\in \RR \times \RR^2$, where $h=(h_2,h_3)$.

If $f\in B^\alpha_{1,\infty}$ we set
$$
\|f\|_{B^\alpha_{1,\infty}}:= \|f\|_1 + \sup_{(h_1,h)\neq(0,0)}
\widetilde\rho(h_1,h)^{-\alpha}
\displaystyle\int|f(x_1+h_1,x+h)-f(x_1,x)|\,dx_1dx.
$$
\end{definition}

Our proof hinges on the following result.
\begin{theorem}\label{Besov}
Let $\{\psi_l\}_{l\in \ZZ}$ be a family of functions such that for
some positive constants $C,\alpha,\varepsilon$ the following
hypotheses hold uniformly in $l$:
\begin{itemize}
\item[(i)] $\{\psi_l\}_{l\in \ZZ}\subset L^1(\RR^3)$;
\item[(ii)] $\displaystyle\int |\psi_l(x_1,x)|
(1+\widetilde\rho(x_1,x))^\varepsilon\,dx_1dx \leq C$;
\item[(iii)] $\displaystyle\int \psi_l(x_1,x)\,dx_1dx =0$;
\item[(iv)] $\|\psi_l\|_{B^\alpha_{1,\infty}}\leq C$.
\end{itemize}
Then the series $\sum_{l\in \ZZ}2^{-(m+n+1)l}\psi_l(2^{-l}\bullet \cdot)$ converges in the sense of
distributions to a Calder\'{o}n-Zygmund kernel.
\end{theorem}

Finally we can state and prove the  $L^p$-bounds for $T_2$.

\begin{proposition}\label{L^p_boundedness_of_T_2}
The operator $T_2$ is bounded on $L^p(\RR^3)$ for $1<p<\infty$.
\end{proposition}
\begin{proof}
The proof is by complex  interpolation.

We first prove 
that  $T_{2,z}$ is
bounded on $L^2(\RR^3)$ for    $-\frac{ 1 }{2mn^2 }<\Re e z<0$.

En easy computation shows that
\begin{align*}
\big(
2^{-(m+n+1)j_1}
\lambda_J 
(2^{-j_1}\cdot,
(2^{-2mn\,j_1}\circ \cdot)\big)\,{\hat{}}\,(\xi_1, \xi)&=
\widehat{\lambda_J}\big( 2^{j_1}\xi_1, 2^{2mn j_1}\circ \xi\big)\\
&=
 \widehat{B^z}\big(
2^{2mn j_1}\circ \xi\big)\,m_J (\xi_1, \xi)\,,
\end{align*}
where
$\lambda_J$ id defined by 
(\ref{lambdaJ}),
 $m_J$ is the Fourier transform of
$2^{-j_1-jQ}
\varphi_J\left(2^{-j_1}x_1,2^{-j}\circ (x-\gamma(x_1)\right)$ and it is given by
(\ref{Fourier_transform}), while $ \widehat{B^z}$ is the Fourier
transform of
$2^{-(m+n)j_1}
\Big(
\delta_0\otimes
B^z
\big( 
\xi_1,
2^{-2mn j_1}\circ x\big)\Big)$. 

It follows from
  Proposition
\ref{FouriertransformI^z}
that
\begin{equation}\label{decadimento B^z}
\Big| \widehat{B^z} (\xi)\Big|\le C(1+\rho (\xi))^{-\Re e z}\quad. 
\end{equation}
Our aim is now to prove that
the series
$$\sum_{2mn j_1>j}
\widehat{B^z}\big(
2^{2mn j_1}\circ \xi\big)\,m_J (\xi_1, \xi)\,,$$ 
corresponding to the Fourier transform of (\ref{K_2z}),
converges boundedly to a bounded function.
The inequality (\ref{decadimento B^z}) yields
\begin{align*}\Big|
\sum_{2mn j_1>j}
 \widehat{B^z}\big(
2^{2mn j_1}&\circ \xi\big)\,m_J (\xi_1, \xi)\Big|\le
C_z 
\sum_{\tiny{\begin{array}{ll}
2^{j_1}|||(\xi_1,\xi)||| \le 1\\
2mn\,j_1> j\\
\end{array}}}
\big( 1+\rho (2^{2mn j_1}\circ \xi)\big)^{|\Re e z| }\,\big| m_J (\xi_1, \xi)\big|\\
&+
\sum_{\tiny{\begin{array}{ll}
2^{j_1}|||(\xi_1,\xi)||| > 1\\
2mn\,j_1> j\\
\end{array}}}
\big( 1+\rho (2^{2mn j_1}\circ \xi)\big)^{|\Re e z| }\,\big| m_J (\xi_1, \xi)\big|=:\mathrm{J_1}+\mathrm{J_2}\,.
\end{align*}
Since
$$2^{2mn j_1} \rho (\xi)=
2^{2mn j_1}
(\xi_2^{2n}+\xi_3^{2m})\le
2\cdot\big( 2^{ j_1}
|||(\xi_1,\xi)|||\big)^{2mn}\,,
$$
we have
\begin{equation}
\mathrm{J_1}\le
C\,3^{|\Re e z| }
\sum_{\tiny{\begin{array}{ll}
2^{j_1}|||(\xi_1,\xi)||| \le 1\\
2mn\,j_1> j\\
\end{array}}}
\,\big| m_J (\xi_1, \xi)\big|\le C_z\,,
\end{equation}
In the light of what has been proved in Proposition \ref{L^2_boundedness},
to estimate 
$\mathrm{J_2}$  observe that
$$\mathrm{J_2}
\le C_z
\sum_{\tiny{\begin{array}{ll}
2^{j_1}|||(\xi_1,\xi)||| > 1\\
2mn\,j_1> j\\
\end{array}}}
\big( 2^{ j_1}
|||(\xi_1,\xi)|||\big)^{2mn|\Re e z|}\big| m_J (\xi_1, \xi)\big|\,.
$$
Now, by using
(\ref{link_between_multiplier_and_I_J})
and estimate (\ref{estimate_3}) for $I_J$, we obtain 
$$\mathrm{J_2}
\le C_z
\sum_{\tiny{\begin{array}{ll}
2^{j_1}|||(\xi_1,\xi)||| > 1\\
2mn\,j_1> j\\
\end{array}}}
\frac{
\big( 2^{ j_1} |||(\xi_1,\xi)|||\big)^{2mn|\Re e z|}}
{ \big(2^{ j_1 }|||(\xi_1,\xi)|||\big)^{1/n} }
\frac{ \rho (2^j\circ \xi)^{ 1/2n} }{\Big( 1+\rho (2^j\circ \xi)\Big)^N}\,,
$$
so 
that
  the operator $T_{2,z}$ is bounded on $L^2(\RR^3)$
   if 
$-\frac{ 1 }{2mn^2 }<\Re e z<0$.



We will now show that the operator $T_{2,z}$ is
bounded on $L^p(\RR^3)$ for $1<p<\infty$
 for  $0<\Re ez<Q$.


By setting
$j-2mn\,j_1=k$
and using 
(\ref{2mn}),
(\ref{K_2z}) may be written as
$$
\begin{aligned}
K_{2,z}(x_1,x) &= 
\sum_{j_1 \in \ZZ}2^{-(m+n+1)j_1} \sum_{k=-\infty}^0
\lambda_{(j_1,k+2mn\,j_1)}(2^{-j_1}\bullet(x_1,x))
\end{aligned}
$$
where
\begin{equation}   \label{lambda_functions}
\lambda_{(j_1,k+2mn\,j_1)}(x_1,x) = 2^{-Qk} \displaystyle\int
\varphi_{(j_1,k+2mn\,j_1)}(x_1,2^{-k}\circ(x-u-\gamma(x_1)))B^z(u)\,du.
\end{equation}
We shall now prove that the functions 
$$
\psi_{j_1}(x_1,x):=\sum_{k=-\infty}^0 \lambda_{(j_1,k+2mn\,j_1)}(x_1,x)
$$
satisfy the hypotheses of Theorem \ref{Besov} uniformly in $j_1$.

Since the estimates on the functions $\psi_{j_1}(x_1,x)$ that we
will prove later  are independent of $j_1$, it suffices to prove that
$$
\psi_0(x_1,x) = \sum_{k=-\infty}^0 \lambda_{(0,k)}(x_1,x)
$$
satisfies the hypotheses of Theorem \ref{Besov}.
\par
We begin proving  that $\psi_0$ belongs to $L^1(\RR^3)$.
As a consequence of  (\ref{x_moments}) with $\beta=0$ 
and a change of variable we obtain
$$\psi_0 (\xu,x)= \sum_{k=-\infty}^0 
\int \varphi_{(0,k)}
(x_1,v) \big( B^{z} (x-2^k \circ v-\gamma (x_1))
-B^z (x-\gamma (x_1))\big) dv\,,$$
whence
\begin{align*}
\int |\psi_0 (\xu,x)|d\xu dx
&\le 
\sum_{k=-\infty}^0 
\int 
\Big(
\int
|\varphi_{(0,k)}
(x_1,v)|\cdot \big|  B^{z} (x-2^k \circ v-\gamma (x_1)) \\
&\qquad\qquad -B^z (x-\gamma (x_1))\big|dv \Big) d\xu dx\\
&=\sum_{k=-\infty}^0 
\int 
|\varphi_{(0,k)}
(x_1,v)|
\big(
\int 
 \big|  B^{z} (x-2^k \circ v)\\
&\qquad\qquad-B^z (x)\big|dx  \big)
 d\xu dv\\
&\le C_z 
\sum_{k=-\infty}^0 
\int 
\big|\varphi_{(0,k)}
(x_1,v)\big|
\rho
(2^k \circ v)^{\Re e z}
d\xu dv   \\
&= C_z 
\sum_{k=-\infty}^0 
2^{k\Re e z}
\int
\big|\varphi_{(0,k)}
(x_1,v)\big|
\rho
(v)^{\Re e z}
d\xu dv   
\le C_z\,,
\end{align*}
by  Lemma \ref{lemmaBz} and since 
$\{\varphi_{(0,k)}\}$ is a bounded set in
$\Sch (\RR^3)$.

Now we shall prove that $\psi_0$ 
satisfies the hypothesis ii)  of Theorem \ref{Besov}
with $\varepsilon =1$ and
\begin{equation}
\widetilde
\rho(x_1,x)=|x_1|+|x_2|^\frac1m+|x_3|^\frac1n\,.
\end{equation}
We split the integral as
\begin{align*}
\int
 |\psi_0 (\xu,x)|
(1+\widetilde\rho(x_1,x))
d\xu dx&= 
\Big(
\int_{\tiny{\begin{array}{ll}
{\unme \le |\xu|\le 4}\\
{|x_2|^\frac1m+|x_3|^\frac1n\le M}
\end{array}}} 
+
\int_{\tiny{\begin{array}{ll}
{\unme \le |\xu|\le 4}\\
{|x_2|^\frac1m+|x_3|^\frac1n> M}
\end{array}}}  \Big)
 |\psi_0 (\xu,x)|\\
\qquad\qquad\qquad\qquad&\qquad\qquad\quad
(1+\widetilde\rho(x_1,x))\, d\xu dx\\
&=\Im_1+\Im_2\,, \\
\end{align*}
where the costant $M$ will be chosen later.

Now $\Im_1$ is bounded by some positive costant $C$, since
$\psi_0(x_1, x)$ belongs to $L^1 (\RR^3)$
and 
$\widetilde\rho(x_1,x)$ is bounded on the integration set.
To study $\Im_2$, observe that
\begin{align*}
\Im_2
&\le
\sum_{k=-\infty}^0 
2^{-kQ}
\int_{\tiny{\begin{array}{ll}
{\unme \le |\xu|\le 4}\\
{|x_2|^\frac1m+|x_3|^\frac1n> M}
\end{array}}}  
\int_{\rho (u)\le {\frac{1}{4}}}
\Big|
\varphi_{(0,k)}
(x_1,2^{-k} \circ ( x-u-\gamma (x_1))
\Big|
\cdot \Big|  B^{z} (u)\Big|
du \\
&\qquad\qquad\qquad\qquad\qquad
(5+
|x_2|^\frac1m+|x_3|^\frac1n
)\, d\xu dx\\
&\le C
\sum_{k=-\infty}^0 
2^{-kQ}
\int_{\tiny{\begin{array}{ll}
{\unme \le |\xu|\le 4}\\
{|x_2|^\frac1m+|x_3|^\frac1n> M}
\end{array}}}  
\int_{\rho (u)\le {\frac{1}{4}}}
\frac{
||\varphi_{(0,k)}||_{(N)}{\big| B^z (u)\big| }}{
\Big( 1+\big| 2^{-k} \circ ( x-u-\gamma (x_1))
\big|\Big)^N
}
du \\
&\qquad\qquad\qquad\qquad\qquad
(5+
|x_2|^\frac1m+|x_3|^\frac1n
)\, d\xu dx\\
&\le C
\sum_{k=-\infty}^0 
2^{-kQ}
\int_{\tiny{\begin{array}{ll}
{\unme \le |\xu|\le 4}\\
{|x_2|^\frac1m+|x_3|^\frac1n> M}
\end{array}}}  \!\!
(5+
|x_2|^\frac1m+|x_3|^\frac1n
)
\int_{\rho (u)\le {\frac{1}{4}}}
\frac{
||\varphi_{(0,k)}||_{(N)}{\big| B^z (u)\big| }}{
\Big( 1+\big| 2^{-k} \circ  x
\big|\Big)^N
}
du \,d\xu dx\,,\\
\end{align*}
where the last inequality follows from the fact that
for some costant $C>0$  we have
$$
 1+\big| 2^{-k} \circ ( x-u-\gamma (x_1))
\big|\ge
C\big( 1+ \big|2^{-k} \circ x \big|\big)$$
if  
${|x_2|^\frac1m+|x_3|^\frac1n> M}$,
if  $M$ is sufficiently large, 
${\rho (u)\le {\frac{1}{4}}}$, 
${\unme \le |\xu|\le 4}$.

Next  
(\ref{stimerho}) yields
\begin{equation}\label{disugHLP}
\big|
2^{-k} \circ x\big|\ge A\big(
\rho(2^{-k} \circ x)
\big)^{{\frac{1}{2n}}} 
\end{equation}
for some $A>0$, if $k\le 0$ and
$|x_2|^\frac1m+|x_3|^\frac1n>M$.

Now
(\ref{disugHLP}) and the local integrability of
$B^z$ 
give
\begin{align*}\Im_2
&\le C
\sum_{k=-\infty}^0 
2^{-kQ}
\int_{
{|x_2|^\frac1m+|x_3|^\frac1n> M}
}  \!\!
(5+
|x_2|^\frac1m+|x_3|^\frac1n
)
\frac{
1}{
\Big( 1+\rho\big( 2^{-k} \circ  x
\big)\Big)^{\frac{N}{2n}}
}
\,dx\,\\
&\le C
\sum_{k=-\infty}^0 
2^{-kQ}
\int_{
{|x_2|^\frac1m+|x_3|^\frac1n> M}
}  \!\!
(5+
|x_2|^\frac1m+|x_3|^\frac1n
)
\frac{
1}{
\Big( 1+2^{-k}\rho(  x)
\Big)^{\frac{N}{2n}}
}
\,dx\,\\
&\le C
\sum_{k=-\infty}^0 
2^{-kQ}
2^{\frac{kN}{2n}}
\int_{
{|x_2|^\frac1m+|x_3|^\frac1n> M}
}  \!\!
(5+
|x_2|^\frac1m+|x_3|^\frac1n
)
\frac{
1}{
 \rho(  x)^{\frac{N}{2n}}
}
\,dx\le C\,,\\
\end{align*}
yielding ii).

Property iii) is an immediate consequence of (\ref{x_moments}).

Finally
we shall 
prove that $\psi_0$ satisfies property iv) in
 Theorem \ref{Besov}.

We start proving that
\begin{equation}\label{disugintegrale2}
\int\big|
\psi_0
(x_1, x+h)-\psi_0 (x_1,x)\big|dx_1\,dx\le
C'_{z}\rho(h)^{\frac{\Re e z}{4mn}}\,, \;\,h\in\RR^2\,.
\end{equation}
Indeed  we have
\begin{align*}
\int\big|
&\psi_0
(x_1, x+h)-\psi_0 (x_1,x)\big|dx_1\,dx\\
&\le
\sum_{k=-\infty}^0 
\int 
\Big(
\int \Big|
\varphi_{(0,k)}
(x_1,v)
\big(
B^{z} (x+h-2^k \circ v-\gamma (x_1))
-B^{z} (x+h-\gamma (x_1))\big)\\
&\qquad\qquad \qquad-\varphi_{(0,k)}
(x_1,v)
\big(
B^{z} (x-2^k \circ v-\gamma (x_1))
-B^{z} (x-\gamma (x_1))\big)\Big| dx dx_1\Big)dv\\
&=:
\sum_{k=-\infty}^0 \int J_k (v)dv\,.
\end{align*}
Now
observe that
\begin{align}
J_k (v)&\le 2
\int \Big|
\varphi_{(0,k)}
(x_1,v)
\big(
B^{z} (x-2^k \circ v-\gamma (x_1))
-B^{z} (x-\gamma (x_1))\big)\Big| dx dx_1\notag\\
&=2
\int \big|
\varphi_{(0,k)}
(x_1,v)\big|
\Big(\int\Big|
B^{z} (y-2^k \circ v)
-B^{z} (y)\Big|dy \Big) dx_1 \notag\\
&\le C_{z,N}
\frac{\rho (2^k\circ v)^{\Re e z}}{(1+|v|)^N}=
C_{z,N}2^{k\Re e z}
\frac{ \rho (v)^{\Re e z}}{(1+|v|)^N}\,,\;\;N\in\NN\,,\label{primastimaJ}
\\\notag
\end{align}
where we used both the 
Lemma \ref{lemmaBz} and the fact that the functions
$\varphi_{(0,k)}
$ are a bounded set in $\Sch (\RR^3)$.

On the other hand, 
as a consequence of Lemma
 \ref{lemmaBz}
\begin{align}
J_k (v)&\le 2
\int \Big|
\varphi_{(0,k)}
(x_1,v)
B^{z} (y+h-\gamma (x_1))
-\varphi_{(0,k)}
(x_1,v)
B^{z} (y-\gamma (x_1))
\Big|dy dx_1\notag \\
&\le 
C_{z,N}
\frac{ \rho (h)^{\Re e z}}{(1+|v|)^N}\,,\;\;N\in\NN\,,h\in\RR^2\,.
\label{secondastimaJ}\end{align}

By taking a suitable  mean
 between 
 estimates
 (\ref{primastimaJ})
and
 (\ref{secondastimaJ})
we obtain
\begin{equation}
J_k (v)=J_k(v)^{\frac{4mn-1}{4mn}}
\cdot J_k(v)^{\frac{1}{4mn}}\le 
C_{z,N}
2^{{\frac{4mn-1}{4mn}}{ k\Re e z}}
\rho (h)^{\frac{\Re e z}{4mn}}
\frac{ \rho (v)^{{\frac{4mn-1}{4mn}}{\Re e z}}}{(1+|v|)^N}\,,\;\;N\in\NN\,,h\in\RR\,,
\end{equation}
so that 
\begin{align*}
\int\big|
\psi_0
(x_1, x+h)&-\psi_0 (x_1,x)\big|dx_1\,dx\le
\sum_{k=-\infty}^0 \int J_k (v)dv\\
&\le
C_{z,N}
\sum_{k=-\infty}^0
2^{{\frac{4mn-1}{4mn}}{ k\Re e z}}
\rho (h)^{\frac{\Re e z}{4mn}}
\int
\frac{ \rho (v)^{\frac{4mn-1}{4mn}{\Re e z}}}{(1+|v|)^N}
dv
\le 
C_{z,N}
\rho (h)^{\frac{\Re e z}{4mn}}\,,
\\
\end{align*}
proving (\ref{disugintegrale2}).

In an analogous way we can prove 
 the inequality
\begin{equation}\label{disugintegrale1}
\int\big|
\psi_0
(x_1+h_1, x)-\psi_0 (x_1,x)\big|dx_1\,dx\le
C'_{z}|h_1|^{\frac{\Re e z}{2}}\,.
\end{equation}

Finally we have
\begin{align*}
&\int\big|
\psi_0
(x_1+h_1, x+h)-\psi_0 (x_1,x)\big|dx_1\,dx\\
&\le\int\big| 
\psi_0
(x_1+h_1, x+h)
-\psi_0
(x_1+h_1, x)\big|dx_1\,dx+
\int\big| 
\psi_0
(x_1+h_1, x)
-\psi_0
(x_1, x)\big|dx_1\,dx
\\
&\le C_{z} \Big(
\big|h_1\big|^{\frac{\Re e z}{2}}
+
\rho(h)^{\frac{\Re e z}{4mn}}\Big)
\end{align*}
Since
$$ \big|h_1\big|^{\frac{\Re e z}{2}}
+
\rho(h)^{\frac{\Re e z}{4mn}}
\le
C \Big(
\big|h_1\big|
+
\rho(h)^{\frac{1}{2mn}}
\Big)^{\frac{\Re e z}{2}}
$$
and, by a  standard inequality, 
$$
{\big( h_2^{2n}+h_3^{2m}\big)}^{\frac{1}{2mn}}\le 
|h_2|^\frac1m+|h_3|^\frac1n\,,
$$
we finally get
$$ \big|h_1\big|^{\frac{\Re e z}{2}}
+
\rho(h)^{\frac{\Re e z}{4mn}}
\le
C 
\tilde{\rho}(h_1,h)^{\frac{\Re e z}{2}}\,,
$$
proving (iv) with $\alpha=
{\frac{\Re e z}{2}}$.
As a consequence of  Theorem \ref{Besov} the operator
  $T_{2,z}$ is
bounded on $L^p(\RR^3)$, $1<p<\infty$,
 for  $0<\Re ez<Q$. 

Finally choose  $p_0\in (1,2)$ and
fix $q_0\in (1, p_0)$
such that
$$\frac{1}{p_0}=
\frac{b}{2(b-a)}-\frac{a}{q_0 (b-a)}\,.$$
Then
the operator
$T_{2,0}=T_2$ is
bounded on $L^{p_0}
(\RR^3)$. 
By the arbitrariness of $p_0$ and by
duality we conclude that 
$T_{2,0}=T_2$ is
bounded on $L^{p}
(\RR^3)$ for all $1<p<\infty$. 
\end{proof}

\section{Final remarks}\label{final}\medskip


 \subsubsection*{Remark 1.}
 We observe that our results also hold in the more general situation in which
 the curve $\gamma(x_1) = (x_1^m, x_1^n)$ is perturbed to $\hat \gamma(x_1) = (x_1^m + \lambda_2 (x_1), x_1^n + \lambda_3(x_1))$
 with $\lambda_1$ and $\lambda_2$ smooth and satisfying $\lambda_2(x_1) = \it{o}(x_1^m)$ and $\lambda_3(x_1) = \it{o}(x_1^n)$.
In fact, given  a product kernel $K_0$ on $\RR^3$, we define the distribution $K$ by
\begin{equation}\notag \label{adapted_product_kernel_defi}
\displaystyle \int K(x_1,x)f(x_1,x)\,dx_1dx := \int
K_0(x_1,x)f(x_1,x+ \hat\gamma(x_1))\,dx_1dx
\end{equation}
for all Schwartz functions $f$ on $\RR^3$.
Then we break the integral
on the right-hand side
as
\begin{equation}
\int
K_0(x_1,x)f(x_1,x+ \hat \gamma(x_1))\,dx_1dx=
\left(\int\limits_{\Bbb R} \int\limits_{-1}^1
+\int\limits_{\Bbb R} \int\limits_{|x_1|>1}\right)
K_0(x_1,x)f(x_1,x+\hat \gamma(x_1))\,dx_1dx\,.\notag
\end{equation}
The first term, by means of a Taylor expansion, may be reduced 
to the polynomial case, while the latter one 
is of Calderon-Zygmund type.

\subsubsection*{Remark 2.}
Our results still hold 
for convolution operators 
on $\RR^d=\RR\times\RR^{d-1}$
with kernels adapted to curves of the form
 $\gamma(x_1)=(x_1^{m_1}, x_1^{m_2}, \ldots, x_1^{m_{d-1}})$,
 $1<m_1< m_2< \ldots< 
m_{d-1}$, $m_j\in\NN$, $j=1,\ldots, d-1$,
with values in $\RR^{d-1}$.
 Since notation in the  higher dimensional case
is more cumbersome, we gave full details of the proof 
only for the space $\RR^3$.

On
$\RR^{d-1}$ we introduce the dilations given by
\begin{equation}  \label{dilations_in_R^d}
\delta \circ x = (\delta^{{1}/{2m_{d-1}}}x_2,\delta^{{1}/{2m_{d-2}}}x_3,\ldots,
\delta^{{1}/{2m_{1}}})\,,\qquad 
\end{equation}
where
$\delta>0$ and
$x=(x_2,\ldots,x_{d-1})$.
Moreover we equip $\RR^{d-1}$
with the  smooth
homogeneous norm 
$$\rho(x)= x_2^{2m_{d-1}} +x_3^{2m_{d-2}}+\ldots +x_ {d}^{2m_{1}}\,$$
on $\RR^{d-1}$.
From this point on, 
the proof follows the same pattern as in $\RR^3$.
In particular,
the strategy of using Bernstein-Sato polynomials 
to build the meromorphic continuation of
the non isotropic Riesz potentials $\rho(x)^{z-\tilde{Q}}$
(here $\tilde{Q}=\frac{1}{2m_{d-1}}+
\frac{1}{2m_{d-2}}
+\ldots+\frac{1}{2m_{1}}$)
 works in the multidimensional case as well, with  some additional 
 computational difficulties 
 in finding the zeros of the Bernstein-Sato polynomials
 \cite{BMS3}.

 \subsubsection*{Remark 3.}
As an example of the class of operators
studied in this paper we exhibit
the following operator, arising in the
study of
the $L^p-L^q$ boundedness 
of a double analytic family 
of fractional integrals along curves in the space
(see \cite{CSe} for the planar case).

Let $\psi$ be a smooth function on $\RR^2$, such that
$\psi(u_1,-u_2)=\psi(u_1,u_2)$ for every $(u_1,u_2) \in \RR^2$,
$\psi\equiv 1$ on $B_{\unme}$ and $\psi\equiv 0$ outside $B_1$, with
$0\le \psi\le 1$ on $\RR^2$
(here
 $B_{r}$, $r>0$, denotes the non-isotropic ball in $\RR^2$ defined by
 (\ref{palla})).
\par
Define an analytic family of distributions $K^{\gamma}_{z}$, for
$\gamma$ and $z$ in $\CC$, $\Re e \gamma\ge 0$, in the following way
\begin{equation}
\label{KGZ} <K^{\gamma}_{z},f>:= \int <D_{z}(u_1,\ud), f(t, \uu
t^m,\ud t^n)
>|t|^{\gamma}\frac{dt}{t},
\end{equation}
where $D_z$, $\Re e\, z>0$,
denotes the
 family of analytic distributions given by
\begin{align*}
<D_{z}, h>:=& <\psi(\cdot,\cdot)\,I^z(\cdot,\cdot)\,,h(\cdot
+1,\cdot+1)>
\\
=& C_z\,\int_{\RR^2} \rho(\uu-1,\ud-1)^{z-Q}\psi\left(
\uu-1,\ud-1\right)\,h(\uu,\ud)\,d\uu,d\ud\,,
\end{align*}
with $C_z:= \frac{G(0)}{2\sigma(S) \Gamma(z)G(z)}$ and $h\in
\ciC^{\infty}_{c}(\RR^2)$. It is straightforward to check that
$D_{z}$ may be extended to all $z\in\CC$.

As a consequence of Proposition \ref{deltaDirac} we have
\begin{equation}\label{deltazero}
<D_0,h>=h(1,1).
\end{equation}

We remark that, if $\Re e \gamma=0$, then
$$
  <K^{\gamma}_{z},f>:=
\displaystyle\lim_{\varepsilon\to 0}\int <D_{z}(\uu,\ud), f(t, \uu
t^m,\ud t^n)>|t|^{i\rho +\varepsilon}\frac{dt}{t} \,,
$$
where $\Im m\gamma=\rho$, for every $f\in \ciC^{\infty}_{c}(\RR^2)$.
Observe moreover that $K^{\gamma}_{z}$ depends analytically on both
$\gamma$ and $z$.

At this point we may introduce the family of convolution operators
with kernel $K^{\gamma}_{z}$ defined by $(\ref{KGZ})$,
   that is
   \begin{equation}
\label{SGZ}
 \left(S^{\gamma}_{z}f\right)(\xu,\xd,\xt):=
\left(K^{\gamma}_{z}*f\right)(\xu,\xd,\xt)= \int <D_{z}(\uu,\ud),
f(\xu-t,\xd- \uu t^m,\xt-\ud t^n)>|t|^{\gamma}\frac{dt}{t}.
\end{equation}

We observe that, 
in the light  of $(\ref{deltazero})$, we have
$$ \left(S^{\gamma}_{0}f\right)(\xu,\xd,\xt):=C_0
 \int_{\RR^3}
f(\xu-t,\xd- t^m,\xt- t^n)>|t|^{\gamma}\frac{dt}{t},$$ that is, at
the height $z=0$ we recover the fractional integration operator
along the curve $t\mapsto \left( t,t^m,t^n\right)$ in the space.

 In a  forthcoming paper \cite{CCiSe}
 we shall give a complete picture of the characteristic set
 of the operator
$ S^{\gamma}_{z}$. A key step in the proof of that result is  the fact 
that
at the height $\Re e z=0$ and for $\Re e\gamma=0$ the kernel
$K^{\gamma}_{z} $
is a product kernel adapted to the curve
$\xu\mapsto \left(\xu^m,\xd^n\right)$.

\end{document}